\documentclass[11pt,reqno,tbtags,a4paper]{amsart}
\usepackage{amssymb}
\usepackage{xpunctuate}
\usepackage{url}
\usepackage[square,numbers]{natbib}
\bibpunct[, ]{[}{]}{;}{n}{,}{,}

\title[patterns in a random tree or forest permutation]
{The number of occurrences of patterns in a random tree or forest permutation}

\date{8 March, 2022; typo corrected 9 March, 2022}

\author{Svante Janson}
\thanks{Supported by the Knut and Alice Wallenberg Foundation}
\address{Department of Mathematics, Uppsala University, PO Box 480,
SE-751~06 Uppsala, Sweden}
\email{svante.janson@math.uu.se}
\newcommand\urladdrx[1]{{\urladdr{\def~{{\tiny$\sim$}}#1}}}
\urladdrx{http://www2.math.uu.se/~svante/}

\subjclass[2020]{} 

\numberwithin{equation}{section}

\renewcommand\le{\leqslant}
\renewcommand\ge{\geqslant}

\allowdisplaybreaks

\theoremstyle{plain}
\newtheorem{theorem}{Theorem}[section]
\newtheorem{lemma}[theorem]{Lemma}
\newtheorem{prop}[theorem]{Proposition}
\newtheorem{corollary}[theorem]{Corollary}

\theoremstyle{definition}

\newtheorem{exampleqqq}[theorem]{Example}
\newenvironment{example}{\begin{exampleqqq}}
  {\hfill\qedsymbol\end{exampleqqq}}

\newtheorem{remarkqqq}[theorem]{Remark}
\newenvironment{remark}{\begin{remarkqqq}}
  {\hfill\qedsymbol\end{remarkqqq}}

\newtheorem{definition}[theorem]{Definition}
\newtheorem{problem}[theorem]{Problem}

\theoremstyle{remark}

\newenvironment{romenumerate}[1][-10pt]{
\addtolength{\leftmargini}{#1}\begin{enumerate}
 \renewcommand{\labelenumi}{\textup{(\roman{enumi})}}%
 }{\end{enumerate}}

\newenvironment{PXenumerate}[1]{
\begin{enumerate}
 \renewcommand{\labelenumi}{\textup{(#1\arabic{enumi})}}%
 }{\end{enumerate}}

\newcounter{oldenumi}
{\setcounter{oldenumi}{\value{enumi}}
\begin{romenumerate} \setcounter{enumi}{\value{oldenumi}}}
{\end{romenumerate}}

\newcounter{thmenumerate}

\newcounter{xenumerate}   

\newcommand\pfitem[1]{\par(#1):}

\newcommand\pfcase[1]{\smallskip\noindent\refstepcounter{steps}%
  \emph{Case \arabic{steps}: #1.} \noindent}
\newcommand\resetCase{\setcounter{steps}{0}}
\newcounter{steps}

\newcommand{\refT}[1]{Theorem~\ref{#1}}
\newcommand{\refTs}[1]{Theorems~\ref{#1}}
\newcommand{\refC}[1]{Corollary~\ref{#1}}

\newcommand{\refL}[1]{Lemma~\ref{#1}}
\newcommand{\refLs}[1]{Lemmas~\ref{#1}}
\newcommand{\refR}[1]{Remark~\ref{#1}}
\newcommand{\refRs}[1]{Remarks~\ref{#1}}
\newcommand{\refS}[1]{Section~\ref{#1}}
\newcommand{\refSs}[1]{Sections~\ref{#1}}
\newcommand{\refSS}[1]{Section~\ref{#1}}
\newcommand{\refP}[1]{Proposition~\ref{#1}}
\newcommand{\refPs}[1]{Propositions~\ref{#1}}
\newcommand{\refProb}[1]{Problem~\ref{#1}}

\newcommand{\refE}[1]{Example~\ref{#1}}

\newcommand{\refTab}[1]{Table~\ref{#1}}

\newcommand\marginal[1]{\marginpar[\raggedleft\tiny #1]{\raggedright\tiny#1}}
\newcommand\REM[1]{{\raggedright\texttt{[#1]}\par\marginal{XXX}}}
\newcommand\XREM[1]{\relax}

\begingroup
  \divide\count255 by 60
  \multiply\count255 by -60
  \advance\count255 by \time
  \ifnum \count255 < 10 \xdef\klockan{\the\count1.0\the\count255}
  \else\xdef\klockan{\the\count1.\the\count255}\fi
\endgroup

\DeclareMathOperator*{\sumx}{\sum\nolimits^{*}}

\newcommand{\sumxi}{\sum'}

\newcommand{\sumko}{\sum_{k=0}^\infty}
\newcommand{\summo}{\sum_{m=0}^\infty}
\newcommand{\sumno}{\sum_{n=0}^\infty}

\newcommand{\sumk}{\sum_{k=1}^\infty}
\newcommand{\suml}{\sum_{\ell=1}^\infty}

\newcommand{\sumn}{\sum_{n=1}^\infty}
\newcommand{\sumim}{\sum_{i=1}^m}
\newcommand{\sumin}{\sum_{i=1}^n}

\newcommand{\prodib}{\prod_{i=1}^b}
\newcommand{\prodim}{\prod_{i=1}^m}

\newcommand\set[1]{\ensuremath{\{#1\}}}
\newcommand\bigset[1]{\ensuremath{\bigl\{#1\bigr\}}}

\newcommand\xpar[1]{(#1)}
\newcommand\bigpar[1]{\bigl(#1\bigr)}
\newcommand\Bigpar[1]{\Bigl(#1\Bigr)}

\newcommand\lrpar[1]{\left(#1\right)}

\newcommand\bigsqpar[1]{\bigl[#1\bigr]}
\newcommand\Bigsqpar[1]{\Bigl[#1\Bigr]}

\newcommand\lrsqpar[1]{\left[#1\right]}
\newcommand\xcpar[1]{\{#1\}}

\newcommand\bigabs[1]{\bigl\lvert#1\bigr\rvert}

\def\rompar(#1){\textup(#1\textup)}    

\def\xexp(#1){e^{#1}}

\newcommand\ntoo{\ensuremath{{n\to\infty}}}

\newcommand\xtoo{\ensuremath{{x\to\infty}}}

\newcommand\punkt{\xperiod}    
\newcommand\iid{i.i.d\punkt}    
\newcommand\ie{i.e\punkt}
\newcommand\eg{e.g\punkt}

\newcommand\cf{cf\punkt}
\newcommand{\as}{a.s\punkt}
\newcommand{\aex}{a.e\punkt}

\newcommand{\tend}{\longrightarrow}
\newcommand\dto{\overset{\mathrm{d}}{\tend}}
\newcommand\pto{\overset{\mathrm{p}}{\tend}}

\newcommand\eqd{\overset{\mathrm{d}}{=}}

\newcommand\OLX{\OL{*}}
\newcommand\OLq{\OL{q}}
\newcommand\OL[1]{O_{L^{#1}}}

\newcommand\bbR{\mathbb R}

\newcounter{CC}
\newcounter{cc}

\newcommand\E{\operatorname{\mathbb E{}}}
\renewcommand\P{\operatorname{\mathbb P{}}}

\newcommand\Var{\operatorname{Var}}
\newcommand\Cov{\operatorname{Cov}}

\newcommand\Ge{\operatorname{Ge}}

\newcommand\ga{\alpha}

\newcommand\gd{\delta}

\newcommand\gf{\varphi}

\newcommand\gl{\lambda}

\newcommand\gO{\Omega}
\newcommand\gs{\sigma}
\newcommand\gS{\Sigma}
\newcommand\gss{\sigma^2}
\newcommand\gt{\tau}
\newcommand\gu{\upsilon}

\newcommand\cS{{\mathcal S}}

\newcommand\indic[1]{\boldsymbol1\xcpar{#1}}

\newcommand\qw{^{-1}}
\newcommand\qww{^{-2}}
\newcommand\qq{^{1/2}}

\newcommand\dd{\,\mathrm{d}}
\newcommand\ddx{\mathrm{d}}

\newcommand{\pgf}{probability generating function}

\newcommand\N{\mathsf N}

\newcommand\xoo{_1^\infty}

\newcommand\fF{\mathfrak{F}}
\newcommand\fS{\mathfrak{S}}
\newcommand\fT{\mathfrak{T}}
\newcommand\fSx{\fS_*}
\newcommand\Ustat{$U$-statistic}
\newcommand\sfC{\mathsf{C}}
\newcommand\sfL{\mathsf{L}}
\newcommand\sfR{\mathsf{R}}
\newcommand\bpi{\boldsymbol{\pi}}
\newcommand\btau{\boldsymbol{\tau}}
\newcommand\btaux{\boldsymbol{\tau}^{\boldsymbol*}}
\newcommand\tbtau{\widetilde{\boldsymbol{\tau}}}
\newcommand\occ{\operatorname{occ}}
\newcommand\ns{\occ_\gs}
\newcommand\nt{\occ_\gt}
\newcommand\npp[1]{\occ_{#1}}
\newcommand\perm[1]{\ensuremath{{#1}}}
\newcommand\opluss{\oplus\dots\oplus}
\newcommand\tautaum{\tau_1\opluss\tau_m}
\newcommand\tbtautaum{\tbtau_1\opluss\tbtau_m}
\newcommand\tbtautauN{\tbtau_1\opluss\tbtau_{N(n)}}

\newcommand\hgO{\widehat{\gO}}

\newcommand\tell{\tilde{\ell}}
\newcommand\tr{\tilde{r}}
\newcommand\ggs{\gs}
\newcommand\ggss{\ggs^2}
\newcommand\gamm{\gamma^2}

\newcommand\tgs{\tilde\gs}
\newcommand\Gtau{G_{|\tbtau|}}
\newcommand\gsik{{\iota_d}}
\newcommand\nuN{N}
\newcommand{\mux}{\nu}
\newcommand\gamx{\gamma}
\newcommand\gamxx{\gamx^2}

\newcommand\tmu{\widetilde{\mu}}
\newcommand\tmuinv{\frac{5+\sqrt5}{10}}
\newcommand\gali{\ga_{\sfL,i}}
\newcommand\gari{\ga_{\sfR,i}}
\newcommand\SNn{S_{N(n)}}
\newcommand\UU{\widehat U}
\newcommand\ff{\overline f}
\newcommand\fgs{\ff_\gs}
\newcommand\gbb{b}
\newcommand\gbt{T}
\newcommand\mb{b}
\newcommand\mbb{b'}
\newcommand\MD{D'}
\newcommand\MDd{D'_d}

\newcommand\bgf{\overline\varphi}

\newcommand{\Holder}{H\"older}

\begin{document}

\begin{abstract} 
The classes of tree permutations and forest permutations were defined by
Acan and Hitczenko (2016).
We study random permutations of a given length from these classes,
and in particular the number of occurrences of a fixed pattern in one of
these random permutations.
The main results show that the distributions of
these numbers are asymptotically normal.

The proof uses 
representations of random tree and forest permutations that
enable us to express the number of occurrences of a pattern by a
type of $U$-statistics; we then use general limit theorems for the latter.
\end{abstract}

\maketitle

\section{Introduction}\label{S:intro}

A number of authors have studied  properties of random permutations drawn
uniformly from all permutations of a given (large) length
in some given class of permutations.
The chosen class of permutations is often
a \emph{pattern class}, \i.e.,
is the class of all permutations avoiding a certain set of
one or several given patterns;
equivalently, the class is closed under taking
patterns (subpermutations).
(See \refS{Snot} for definitions of various terms used here and below.)
Several different properties have been studied; 
in the present paper we consider the asymptotic
distribution of the  number of occurences of some fixed pattern.
For this problem (and many others),
it seems impossible to give general results valid for all such 
permutation classes.
(See \eg{} \citet{Pak} for some related impossibility results supporting
this.)
Therefore, typically these classes are studied one by one,
with methods depending on the knowledge of some structure theorem for
permutations in that particular class.
See \eg{} \cite{BassinoEtAl} and \cite{SJ333} for some results of this type.

The present paper continues this line of research by studying the number of
occurences of a given pattern in a random \emph{tree permutation}
or \emph{forest permutation}.
These classes of permutations were defined by \citet{AH} as follows.
\begin{definition}\label{D1}
For a permutation $\pi$ of $[n]$, its permutation graph
$G_\pi$ 
is the
(labelled, undirected) graph with vertex set $[n]$, and an 
edge $ij$ for every inversion
$(i,j)$ in $\pi$, \ie, for every pair $(i,j)$ such that  $i<j$ and
$\pi(i)>\pi(j)$.

  A permutation $\pi$ is 
a \emph{tree permutation} if $G_\pi$ is a tree, 
and
a \emph{forest permutation} if the graph $G_\pi$ is a forest
(i.e., acyclic).
\end{definition}

Thus, every tree permutation is a forest permutation.

\citet{AH} noted also the following characterization, showing
that the forest permutations form a pattern class.

\begin{prop}[\cite{AH}]\label{P1}
 The forest permutations are precisely the permutations avoiding the
patterns $321$ and $3412$.
\end{prop}

However, the class of tree permutations is \emph{not} a pattern class, since
a subpermutation of a tree permutation may be a forest permutation with a
disconnected permutation graph. (For example, 312 is a tree permutation, but
its subpermutation 12 is not.)

The structures of tree permutations and forest permutations were studied in
\cite{AH}; see \refS{Sforest}.
Using this, and results on (conditioned) \Ustat{s}, we will 
show that the number of occurences of a fixed pattern in a random tree or
forest permutation is asymptotically normal, as the length tends to $\infty$;
precise results are stated in \refS{Smain}, and proved in the remainder of
the paper.
\refS{SSU} defines the versions of \Ustat{s} that are used in the paper, and
cites some results for them from \cite{SJ332} and \cite{SJ360}.
Tree and forest
permutations are studied in \refSs{Scode}--\ref{Staun}, leading to a
representation of random forest permutations in \refS{SRT}
and a, quite different, representation of random
tree permutations in \refS{SRTB}; these representations both enable us to 
count patterns by \Ustat{s}, which eventually yields proofs of the theorems.

\begin{remark}
  Although we use similar methods for patterns in random tree permutations
  and in random forest permutations, the details are quite different, and we
  see no direct relation between the results for the two cases.
Note that a random forest permutation is a (random) sum of tree
permutations,
but most of these are very small 
(see \eqref{pint} and \eqref{tbtau});
hence there is no reason to expect a relation between asymptotics for large
forest permutations and large tree permutations. 
\end{remark}

\section{Definitions and notation}\label{Snot}

\subsection{Permutations}
Let $\fS_n$ be the set of permutations of $[n]:=\set{1,\dots,n}$,
and $\fSx:=\bigcup_{n}\fS_n$.
Similarly, let $\fF_n$ be the set of all forest permutations of length $n$
and $\fT_n$ the subset of tree permutations,
and let $\fF_*:=\bigcup_{n}\fF_n$ and $\fT_*:=\bigcup_{n}\fT_n$.
Thus
$\fT_n\subseteq\fF_n\subseteq\fS_n$. 

We denote the length of a permutation $\pi$ by $|\pi|$.


\subsection{Occurrence of patterns}
If $\gs=\gs_1\dotsm\gs_m\in\fS_m$ and $\pi=\pi_1\dotsm\pi_n\in\fS_n$,
then an \emph{occurrence} of $\gs$ in $\pi$ is a 
subsequence $\pi_{i_1}\dotsm\pi_{i_m}$, 
with $1\le i_1<\dots<i_m\le n$, 
that
has the same order as
$\gs$, i.e., 
$\pi_{i_j}<\pi_{i_k} \iff \gs_j<\gs_k$ for all $j,k\in [m]$.
In this context, $\gs$ is often called a \emph{pattern};
we may also say that $\gs$ is  a \emph{subpermutation} of $\pi$.
We let $\ns(\pi)$ be the number of occurrences of $\gs$ in $\pi$, and note
that
\begin{equation}\label{11}
  \sum_{\gs\in\fS_m} \ns(\pi) = \binom nm,
\end{equation}
for every $\pi\in\fS_n$ and every $m$. 
For example, an inversion is an occurrence of \perm{21}, and thus
$\npp{21}(\pi)$ is the number of inversions in $\pi$.

We say that a permutation $\pi$ \emph{avoids} another permutation 
$\tau$ if $\nt(\pi)=0$; otherwise, $\pi$
\emph{contains} $\tau$.

\subsection{Sums and decompositions of permutations}
\label{SSblocks}  
If $\gs\in\fS_m$ and $\tau\in\fS_n$,  their (direct) \emph{sum}
$\gs\oplus\tau\in\fS_{m+n}$
is defined
by letting $\tau$ act on $[m+1,m+n]$ in the natural
way; more formally,
$\gs\oplus\tau=\pi\in\fS_{m+n}$ where $\pi_i=\gs_i$ for $1\le i\le m$, and
$\pi_{j+m}=\tau_j+m$ for $1\le j\le n$. 
It is easily seen that $\oplus$ is an associative operation.
We say that a permutation $\pi\in\fS_*$  is
\emph{decomposable} if $\pi=\gs\oplus\tau$ for some 
$\gs,\tau\in\fS_*$, and
\emph{indecomposable} otherwise;
we also call an indecomposable permutation a \emph{block}.
See further \eg{}  \cite[Exercise VI.14]{Comtet}.

It is easy to see that any permutation $\pi\in\fS_*$ has a unique
decomposition $\pi=\pi_1\oplus\dots\oplus\pi_\ell$  into indecomposable
permutations (blocks)  $\pi_1,\dots,\pi_\ell$ 
(for some, unique, $\ell\ge1$);
we may call these the \emph{blocks of $\pi$}

 If $i<j<k$ and $ik$ is an edge in the permutation
graph $G_\pi$
(\ie, an inversion), then at least one of $ij$ and $jk$ is also an edge.
It follows that the components of the graph $G_\pi$ are intervals in $[n]$,
and then it is easy to see that they correspond to the blocks
of $\pi$; in particular, $G_\pi$ is connected if and only if $\pi$ is
indecomposable. 

\subsection{Random permutations}\label{SSRP}
$\btau_n$ will always denote a uniformly random tree permutation of length $n$;
similarly, $\bpi_n$ is a uniformly random forest permutation of length $n$.
In other words, these are uniformly random elements of $\fT_n$ and $\fF_n$,
respectively. 

$\tbtau$ denotes a certain random tree permutation of random length
defined in \refS{SRT}, see \eqref{tbtau}; 
$\tbtau_1,\tbtau_2,\dots$ will denote independent copies of $\tbtau$.
Similarly, 
 $\btaux_m$ is  another random tree permutation of random length,
defined in \refS{SRTB}.

\subsection{Some further notation}\label{SSfurther}

Convergence in distribution is denoted by $\dto$,
and convergence in probability by $\pto$.
We let $\eqd$ denote equality in distribution.

Given  sequences of random variables $X_n$ and constants $a_n>0$, 
and a fixed exponent $q>0$, 
we let
$X_n=\OLq(a_n)$ mean $\E|X_n/a_n|^q=O(1)$.
Moreover, we write
$X_n=\OLX(a_n)$ if 
$X_n=\OLq(a_n)$ for every $q<\infty$.

By ``convergence of all moments'' we mean both ordinary and absolute
moments, including centered versions.

We find it convenient to express some explicit constants using
\begin{align}\label{phi}
  \phi:=\frac{1+\sqrt5}{2},
\end{align}
the golden ratio. Recall that $\phi^2=\phi+1$.
We will also let $p:=\phi\qww$, see \eqref{Tp}--\eqref{p2phi}.

Unspecified limits are as \ntoo.

\section{Main results}\label{Smain}

Our main results are the following; the proofs are given later.
In both cases, note that if $\gs$ is not a forest permutation, then
$\ns(\bpi_n)=0$.
Note also that we may assume $|\gs|\ge2$, since
the case $\gs=1$ is utterly trivial with $\occ_1(\pi)=n$ for every
$\pi\in\gS_n$.
Moreover, if $\tau\in\fT_n$ is a tree permutation, then 
$\occ_{21}(\tau)=n-1$, since the number of inversions equals the number of
edges in the tree $G_\tau$.

\begin{theorem}\label{TT}
Let $\btau_n$ be a uniformly random tree permutation of length $n$,
and let $\gs$ be a fixed forest permutation 
with  block decomposition $\gs=\gs_1\opluss\gs_d$.
Then, as \ntoo,
for some $\gamm=\gamm_\gs\ge0$,
\begin{equation}\label{tt}
\frac{\ns(\btau_n)- n^{d}/d!}
{n^{d-1/2}} 
\dto \N\bigpar{0,\gamx^2},
\end{equation}
with convergence of all moments.
Moreover, $\gamm>0$ unless $|\gs_i|\le2$ for every $i$, \ie, unless 
each block $\gs_i$ is either $1$ or $21$.
\end{theorem}

We state the special case $d=1$ separately.

\begin{corollary}\label{CTT}
Let $\btau_n$ be a uniformly random tree permutation of length $n$,
and let $\gs$ be a fixed tree permutation. 
Then, as \ntoo,
for some $\gamm=\gamm_\gs\ge0$,
\begin{equation}\label{ctt}
\frac{\ns(\btau_n)- n}
{\sqrt n} 
\dto \N\bigpar{0,\gamx^2},
\end{equation}
with convergence of all moments.
Moreover, $\gamm>0$ except in the trivial cases $|\gs|\le2$, 
when $\ns(\btau_n)$ is deterministic ($n$ or $n-1$).
\end{corollary}

Furthermore, when $\gs$ is a tree permutation,
we  give an exact formula for $\E\ns(\btau_n)$ in \refT{TE};
this expectation depends on $n$ and $|\gs|$ only.

The asymptotic variance $\gamm_\gs$ in \refT{TT} and \refC{CTT} can be found
from our proof, but we do not know any explicit formula; we evaluate it
for some simple cases in \refE{Elr}.
Note that \refE{Elr} shows that  $\gamm_\gs$ 
in \refC{CTT} really depends on $\gs$, and, moreover, that it is not simply a
function of $|\gs|$. 

\begin{remark}
  If $\gs$ is a foresst permutation with $d\ge2$ blocks $\gs_i$, all of lengths
  $|\gs_i|\le 2$, then $\gamm=0$ in \eqref{tt}, but $\ns(\btau_n)$ is, in
  general, not deterministic. We conjecture that $\ns(\btau_n)$ is
  asymptotically normal in this case too, with a variance of smaller order
than in \refT{TT}, but we have not pursued this and leave it as an open problem.
 (Cf.\ \refT{TFI} below for random forest permutations $\pi_n$.)
\end{remark}

\begin{problem}\label{Prob=}
Find a combinatorial explanation for 
 the surprising fact that 
the asymptotic expectation $n$ in \eqref{ctt} is the same for all
tree permutations $\gs$.
(We will see in the proof that this is equivalent to the fact that 
the expectation  in \eqref{lj3} is the same for all tree permutations $\gs$.)

More generally, 
find a combinatorial explanation for the  fact that 
the asymptotic expectation $n^d/d!$ (or, equivalently, $\binom nd$)
in \eqref{tt} depends only on the the number of blocks $d$ in $\gs$.

Moreover, as just mentioned, \refT{TE} shows that for two tree permutations
$\gs_1$ 
and $\gs_2$ of the same length, the expectations $\E\occ_{\gs_1}(\btau_n)$ 
and $\E\occ_{\gs_2}(\btau_n)$ are equal for every $n$. 
(This obviously requires $|\gs_1|=|\gs_2|$, since $\ns(\btau_n)=0$ if
$n<|\gs|$.) 
Again, we do not know a simple proof of this fact, although the proof of
\refT{TE} gives a kind of  combinatorial reason.
Also, we do not know whether the equality extends to two forest
permutations with the same length and the same number of blocks.
\end{problem}

We turn to patterns in a random forest permutation.

\begin{theorem}  \label{TF}
Let $\bpi_n$ be a uniformly random forest permutation of length $n$,
and let $\gs$ be a fixed forest permutation with  block decomposition
$\gs=\gs_1\opluss\gs_d$.
Let $\gl$ be the number of blocks 
$\gs_i$ of length $|\gs_i|=1$, and let
\begin{align}\label{tmugs}
\tmu_{\gs}&:
=\frac{1}{d!}(\phi+2)^{\gl-d}\phi^{4d-3\gl-|\gs|}
=\frac{1}{d!}5^{-(d-\gl)/2}\phi^{3d-2\gl-|\gs|}
.\end{align}
Then, for some 
 $\gamm_\gs\ge0$,
\begin{align}\label{tf1}
\frac{ \occ_\gs(\bpi_n)-\tmu_\gs n^d}{n^{d-1/2}}
\dto \N\bigpar{0,\gamm_\gs},
\end{align}
with convergence of all moments.

Furthermore,
$\gamm_\gs>0$
except in the case $\gs=1\cdots d$ 
(the identity permutation with every $|\gs_i|=1$).
\end{theorem}

Again, the asymptotic variance $\gamm_\gs$ can in principle be found from
our proof, but we do not know any explicit formula; see \refR{RFVar} and
\refE{E21f}.  

In the exceptional case $\gs=1\cdots d$, the limit in \eqref{tf1} is 0,
and a different normalization is required.

\begin{theorem}
  \label{TFI}
Let $\gsik$ be the identity permutation $1\cdots d$ for some $d\ge2$.
Then,
for some $\gamm_\gsik>0$,
\begin{align}\label{cf1}
\frac{ \occ_\gsik(\bpi_n)-\binom{n}{d}+ 
\frac{5+\sqrt5}{10(d-2)!}\, n^{d-1}}{n^{d-3/2}}
\dto \N\bigpar{0,\gamm_\gsik},
\end{align}
with convergence of all moments.
\end{theorem}

\begin{remark}\label{RCW}
If we consider several patterns, 
\eqref{tt}, \eqref{ctt}, \eqref{tf1} and \eqref{cf1} extend
to joint convergence to a multi-variate normal limit.
This follows by the same proof, using \refRs{Rmulti} and \ref{Rmulti2}.
We omit the details. 
\end{remark}

\section{Preliminaries on tree and forest permutations}
\label{Sforest}
We recall some facts from (mainly) \cite{AH} (in our notation); for
completeness we sometimes sketch the arguments, but we refer to \cite{AH}
for further details.

Note first that a permutation is determined by its (labelled) permutation
graph, in other words, the mapping $\pi\mapsto G_\pi$ is injective.
Furthermore, the induced subgraphs of $G_\pi$ are the 
inversion graphs of the patterns occuring in $\pi$, up to obvious
relabelling.

In particular, it is easily seen that the only induced cycles in a
permutation graph are $\sfC_3$ and $\sfC_4$ (as unlabelled graphs);
these are the permutation graphs of $321$ and $3412$ 
(and no other permutations), 
which proves \refP{P1}.


Moreover, $G_\pi$ is a forest if and only if its component are trees,
and thus $\pi$ is a forest permutation if and only
its blocks are tree permutations. In other words,
\begin{align}\label{tomtam}
  \pi\in\fF_*\iff \pi=\tautaum
\end{align}
for some (unique) sequence $\tau_1,\dots,\tau_m$ of tree permutations.
(We will find the asymptotic distribution of the number of blocks in a
random forest permutation in \refT{Tblocks}.)

Let $t_n:=|\fT_n|$
be the number of tree permutations of length $n$.
It is shown in \cite{AH} that 
\begin{align}\label{tn}
  t_n
=
  \begin{cases}
    1,&n=1,\\
2^{n-2},&n\ge 2,
  \end{cases}
\end{align}
and thus the corresponding generating fuction $T(z)$ is
\begin{align}\label{Tz}
  T(z):=\sumn t_nz^n = z+\frac{z^2}{1-2z} =\frac{z-z^2}{1-2z},
\qquad |z|<1/2.
\end{align}
As a consequence of \eqref{Tz} and \eqref{tomtam}, 
if $f_n$ is the number of forest permutations of length
$n$ (with $f_0:=1$), then the corresponding generating function is
\begin{align}
  F(z):=\sumno f_nz^n = \frac{1}{1-T(z)}
=\frac{1-2z}{1-3z+z^2}.
\end{align}
The sequence $(f_n)$ is A001519 in \cite{OEIS}
(where many other interpretations are given).

\medskip
In a permutation $\pi$, label the left-to-right maxima by $\sfL$, and the
right-to-left minima by $\sfR$. Thus, $i$ is labelled $\sfL$ 
if $\pi(j)<\pi(i)$ for every $j<i$, \i.e., if there are no inversions $(j,i)$
with $j<i$. In other words,
\begin{align}\label{L}
  \text{$i$ is labelled $\sfL$} \iff
\text{$i$ is the left endpoint of every adjacent edge in $G_\pi$}.
\end{align}
Similarly,
 $i$ is labelled $\sfR$ 
if there are no inversions $(i,j)$ with $j>i$, and
\begin{align}\label{R}
  \text{$i$ is labelled $\sfR$} \iff
\text{$i$ is the right endpoint of every adjacent edge in $G_\pi$}.
\end{align}

Now, let $\pi$ be a forest permutation.
If $i<j<k$, then $ij$ and $jk$ cannot both be edges in $G_\pi$, since
otherwise, $\pi(i)>\pi(j)>\pi(k)$, so $(i,k)$ would also be an inversion,
and thus $G_\pi$ would contain a cycle $ijk$. 
If follows that every $j\in[n]$ is labelled either $\sfL$ or $\sfR$, or
possibly both.

Moreover, \eqref{L}--\eqref{R} imply that $i$ is labelled both $\sfL$ and
$\sfR$ if and only if $i$ is isolated in $G_\pi$.
In a tree permutation $\pi$ with $|\pi|\ge2$, this is impossible.
Thus, 
if $\pi$ is a tree permutation with $|\pi|\ge2$, then
every $i\in[n]$ is labelled $\sfL$ or $\sfR$, but not both.
Each tree permutation $\tau$ with $|\tau|=n\ge2$
may thus be represented by a string $\gO_\tau$ of $n$ symbols $\sfL$ or
$\sfR$.
(The notation in \cite{AH} is different: 
there $W_1$ [$W_0$] denotes the set of $i$ labelled $\sfL$ [$\sfR$] here.)
The first symbol in $\gO_\tau$
is always $\sfL$ and the last is $\sfR$. 
We let 
$\gS_n:=\bigset{\sfL\set{\sfL,\sfR}^{n-2}\sfR}$ be the set of such strings,
so $\gO_\tau\in\gS_n$.
It is shown in \cite{AH} that
the map 
$\tau\mapsto \gO_\tau$ 
is a bijection between 
$\fT_n$ and  $\gS_n$, for every $n\ge2$.
(Note that $|\fT_n|=2^{n-2}=|\gS_n|$ by \eqref{tn}.)
In other words, for $n\ge2$, the tree permutations in $\fT_n$ can be encoded
by the strings in $\gS_n$.

We follow \cite{AH} and define the \emph{blocks} $B_1,\dots,B_{2m}$ of $\gO_\tau$
as the successive runs of $\sfL$ and $\sfR$ in $\gO_\tau$.
Note that since $\gO_\tau$ begins with $\sfL$ and ends with $\sfR$,
there is always an even number of blocks; an
odd-numbered block 
$B_{2l-1}$ is a run of $\sfL$ and an even-numbered block $B_{2l}$ is a run
of $\sfR$.
(Note that we also  use 'block' in a different sense for the block decomposition
of a  permutation into its blocks (components) in \refSS{SSblocks}; 
there should be no risk of confusion since 
the two different meanings of 'block'
appear in different contexts,
and we will not use both at the same time.)
\cite[Lemma 8]{AH}
shows how the edges and vertex degrees in $G_\tau$ can be found 
explicitly from the
code $\gO_\tau$ and the blocks $B_i$.
We summarize this as follows.

\begin{lemma}[\cite{AH}]  \label{Lleaf}
Let $\tau$ be a tree permutation with $|\tau|\ge2$.
Then the pairs of symbols in $\gO_\tau$ that 
correspond to edges in $G_\tau$ (and thus to inversions in $\tau$)
are:
\begin{PXenumerate}e
\item\label{e1} each $\sfL$ and the nearest following $\sfR$;
\item\label{e2} each $\sfR$ and the nearest preceding $\sfL$;
\item\label{e3} The last $\sfL$ in a block $B_{2k-1}$ and the first $\sfR$
  in $B_{2k+2}$. 
\end{PXenumerate}
 The  symbols in $\gO_\tau$ that correspond to leaves 
in $G_\tau$  are the following:
\begin{PXenumerate}{l}
\item\label{il1} every $\sfL$ that is not the last $\sfL$ in its block;
\item the last but one symbol, if that is $\sfL$;
\item every $\sfR$ that is not the first $\sfR$ in its block;
\item\label{il4} the second symbol, if that is $\sfR$.
\end{PXenumerate}
\end{lemma}

\begin{proof}
  As said above, this is 
\cite[Lemma 8]{AH}, in different notation.
(The four cases \ref{il1}--\ref{il4}
correspond to parts (a),(c),(d),(f) in that lemma.)
\end{proof}

If $\gs$ is a tree permutation with $|\gs|\ge2$ such that its code $\gO_\gs$
has $2m$ blocks, we define $\gbb(\gs):=m$; in other words the code of $\gs$
has $\gbb(\gs)$ $\sfL$-blocks and $\gbb(\gs)$ $\sfR$-blocks.
If $|\gs|=1$, we do not define any code $\gO_\gs$, but we define (for later
convenience) $\gbb(\gs):=1$.

\section{Preliminaries on \Ustat{s}}\label{SSU}

A 
\emph{\Ustat} is a random variable of the form
\begin{equation}\label{U}
  U_n
=U_n(f) 
= \sum_{1\le i_1<\dots<i_d\le n} f\bigpar{X_{i_1},\dots,X_{i_d}},
\qquad n\ge0,
\end{equation}
where $X_1,X_2,\dots$ is an \iid{} sequence of random variables with values
in some measurable space $\cS$, and $f:\cS^d\to\bbR$ is a
given measurable function of $d\ge1$ variables.
(It is often assumed that $f$ is a symmetric function; we do not assume this.)
\Ustat{s} were introduced by \citet{Hoeffding};
we will use versions and results from \cite{SJ332} and \cite{SJ360},
see also \cite{SJ333} 
for similar applications to 
pattern occurences in some other pattern classes.

The fundamental central limit theorem for \Ustat{s}, due to
\citet{Hoeffding} in the symmetric case, 
can in the general (asymmetric) case be stated as follow, 
see 
\cite[Theorem 11.20]{SJIII} and 
\cite[Corollary 3.5 and (moment convergence) Theorem 3.15]{SJ332}. 
Assume that the random variables $X_i$ are \iid, let $X$ denote a
generic $X_i$, and define (for a given $f$)
\begin{align}
  \mu&:=\E f(X_1,\dots,X_d), \label{Umu}
\\
f_i(x)&:=\E\bigsqpar{f(X_1,\dots,X_d)\mid X_i=x}, \label{fi}
\\
\ggs_{ij}&:=\Cov \bigsqpar{f_i(X),f_j(X)},\label{ggsij}
\\\label{Ugss}
      \ggss&:=
\sum_{i,j=1}^d
\frac{(i+j-2)!\,(2d-i-j)!}{(i-1)!\,(j-1)!\,(d-i)!\,(d-j)!\,(2d-1)!}\ggs_{ij}
.
\end{align}
Note that $f_i(x)$ in \cite{SJ332,SJ360} is  $f_i(x)-\mu$ in the present
notation.

\begin{prop}[\cite{SJIII,SJ332}]\label{PU}
Suppose that $(X_i)\xoo$ are \iid{} random variables, and that
$\E|f(X_1,\dots,X_d)|^2<\infty$.
Then, 
with the notation in \eqref{Umu}--\eqref{Ugss},
as \ntoo,
\begin{equation}\label{c1}
  \frac{U_n-\binom nd \mu}{n^{d-1/2}} \dto \N\bigpar{0,\ggss}.
\end{equation}
Furthermore, $\ggss>0$ unless $f_i(X)=\mu$
\as{} for $i=1,\dots,d$.

Moreover, if $\E|f(X_1,\dots,X_d)|^p<\infty$ for some $p\ge2$, 
the \eqref{c1} holds with convergence of all moments of order $\le p$.
\qed
\end{prop}

We will need a renewal theory version of \refP{PU}.
In addition to a sequence $(X_i)\xoo$ and a function $f$ as above, let
$h:\cS\to\bbR$ be another measurable function, and
assume (for simplicity) that $h(X_i)\ge0$ a.s.
Define 
\begin{align}
\nu&:=\E h(X_i), \label{nu}
\\ \label{Snh}
S_n& \phantom:
=S_n(h):=\sumin h(X_i),   
\end{align}
and let for each $x>0$
\begin{align}
N(x)&:=\inf\set{N:S_N\ge x}
. \label{Nx}
\end{align}

\begin{remark}\label{RNN}
  The definition \eqref{Nx} agrees with $N_+(x)$ in \cite{SJ333} but
differs slightly from $N_+(x)$ and $N_-(x)$ in \cite{SJ332} and \cite{SJ360};
this does not affect the asymptotic results used here, 
see \cite[Remark 3.19]{SJ360}.
(For integer valued $h$ and integer $x$, as in our application,
$N(x)=N_+(x-1)$.)
We will use results from \cite{SJ332} and \cite{SJ360}; note that
the event $\set{S_k=n \text{ for some }k\ge0}$ equals
\set{S_{N(n)}=n} in the present notation, and $\set{U_{N_-(n)}=n}$ in the
notation of \cite{SJ332} and \cite{SJ360}.
(When we condition on this event in propositions below, 
we tacitly consider only $n$ such that
the event has positive probability.)
\end{remark}

The following results are special cases of
\cite[Theorems  3.11, 3.13(iii) and 3.18]{SJ332} (with somewhat different
notation). 

\begin{prop}[\cite{SJ332}]\label{PUN}
Suppose that 
$(X_i)$ are \iid,
$\E|f(X_1,\dots,X_d)|^2<\infty$, 
and
$h(X)\ge0$ \as, with $\mux:=\E h(X)>0$
and $\E h(X)^2<\infty$.
Then, with notations as above, 
as \xtoo,
\begin{equation}\label{cvtau}
\frac{U_{N(x)}-\mu{\mux}^{-d}{d!}\qw x^{d}}
{x^{d-1/2}} 
\dto \N\bigpar{0,\gamx^2},
\end{equation}
where,
with $\gss$ given by \eqref{Ugss},
\begin{align}\label{cvtau2}
    \gamx^2
&:=
{\mux}^{1-2d}
\ggss
-2\frac{\mu{\mux}^{-2d}}{(d-1)!\,d!}\sum_{i=1}^{d}
\Cov\bigsqpar{f_i(X),h(X)}
+ \frac{\mu^2{\mux}^{-2d-1}}{(d-1)!^2}\Var\bigsqpar{h(X)}.
\end{align}
Moreover,
$\gamx^2>0$ unless  $f_i(X)=\frac{\mu}{\nu} h(X)$ \as{} for $i=1,\dots,d$.
\qed  
\end{prop}

\begin{prop}[\cite{SJ332}] \label{PUN2}
Suppose in addition to the hypotheses in \refP{PUN} that $h(X)$ is
integer-valued.
Then \eqref{cvtau} holds also conditioned on $S_{N(x)}=x$ 
(\cf{} \refR{RNN})
for integers $x\to\infty$.
\qed
\end{prop}

\begin{prop}[\cite{SJ332}] \label{PUNmom}
Suppose in addition to the hypotheses in \refP{PUN} or \ref{PUN2} 
that 
$\E|f(X_1,\dots,X_d)|^p<\infty$ and  $\E|h(X)|^p<\infty$ for every $p<\infty$.
Then the conclusion \eqref{cvtau} holds 
with convergence of all moments.
\qed
\end{prop}

\begin{remark}  \label{R1}
In the special case $d=1$, when the \Ustat{} \eqref{U} is a standard single sum,
\eqref{Umu}--\eqref{Ugss} and \eqref{cvtau2} simplify to $f_1=f$,
$\gss=\gs_{11}=\Var f(X)$, and
\begin{align}\label{rov}
  \gamm&
=\frac{1}{\nu}\gss-2\frac{\mu}{\nu^2}\Cov\bigsqpar{f(X),h(X)}
+ \frac{\mu^2}{\nu^3}\Var h(X)
\notag\\&
=\frac{1}{\nu} \Var \Bigsqpar{f(X)-\frac{\mu}{\nu}h(X)}.
\end{align}
This special case is classical, see \eg{}
\cite[Theorem 4.2.3]{Gut-SRW}.
\end{remark}

\begin{remark}\label{Rmulti}
The results in \refPs{PU}--\ref{PUNmom} hold jointly for several $f$
(possibly with different $d$). This is not stated explicitly in \cite{SJ332}
(except for \eqref{c1}), but it follows by the  same proofs as
in \cite{SJ332} (perhaps, for convenience, using 
the Skorohod coupling theorem \cite[Theorem~4.30]{Kallenberg} and a.s.\
convergence in the proofs). See also \cite{SJ360}.
\end{remark}

\subsection{Constrained \Ustat{s}}
In this subsection we extend some of the results above to constrained
\Ustat{s}, defined as follows. We consider here only a case relevant for the
application in the present paper; for more general definitions and results,
see \cite{SJ360} (with somewhat different notation).

Let, as above, $(X_i)\xoo$ be an \iid{} sequence of random variables in some
measurable space $\cS$.

Let $d\ge1$ and let $\mb_1,\dots,\mb_d$ be given non-negative integers.
(These are regarded as fixed in this subsection.)
Let 
\begin{align}
\mbb_j&:=\mb_j-1,\label{mbb}
\\
D_j&:=\sum_{1}^j \mb_i, \qquad 0\le j\le d,
\\
\MD_j&:=\sum_{1}^j \mbb_i=D_j-j, \qquad 0\le j\le d,\label{MDi}
\\
D&:=D_d=\sum_{i=1}^d\mb_i=\MD_d+d.  
\end{align}
Suppose that $f:\cS^D\to\bbR$ is a measurable function, and
define the \emph{constrained \Ustat}
\begin{align}\label{UU1}
  \UU_n=\UU_n(f)
:=\sum_{i_1,\dots,i_d} 
f\bigpar{(X_{i_1+k})_{k=0}^{\mbb_1},(X_{i_2+k})_{k=0}^{\mbb_2},\dots,
  (X_{i_d+k})_{k=0}^{\mbb_d}}
\end{align}
summing over all $i_1,\dots,i_d$ such that $i_1\ge1$, 
$i_1+\mbb_1<i_2$, $i_2+\mbb_2<i_3$, \dots, $i_{d-1}+\mbb_{d-1}<i_d$,
and $i_d+\mbb_d\le n$.
(We have grouped the arguments of $f$ in \eqref{UU1}, 
using an obvious notation.)
In other words, $\UU_n$ is defined as $U_n$ in \eqref{U},
with $d$ replaced by $D$,
but only summing over $i_1,\dots,i_D$ such that the $\mb_1$ first indices
are consecutive, as well as the next $\mb_2$, and so on.
In particular, in the special case $\mb_1=\dots=\mb_d=1$,
$\UU_n$ equals the unconstrained \Ustat{} $U_n$ in \eqref{U}.

By replacing $i_j$ by $i_j-\MD_{j-1}$ in \eqref{UU1}, we obtain the
alternative formula
\begin{multline}\label{UU2}
  \UU_n
:=\sum_{1\le i_1<i_2<\dots<i_d\le n-\MD_d} 
f\bigpar{(X_{i_1+k})_{k=0}^{\MD_1}, (X_{i_2+k})_{k=\MD_1}^{\MD_2},
\dots,(X_{i_d+k})_{k=\MD_{d-1}}^{\MD_d}}
.\end{multline}

Define, as in \eqref{Umu},
\begin{align}
    \mu=\mu_f&:=\E f(X_1,\dots,X_D). \label{UUmu}
\end{align}
By \eqref{UU2}, the mean of $\UU_n$ is
\begin{align}\label{EUU}
  \E\UU_n = \binom{n-\MD_d}{d}\mu.
\end{align}

\refP{PU} extends to constrained \Ustat{s} as follows.

\begin{prop}[\cite{SJ360}]\label{PUU1}
  Let $\UU_n=\UU_n(f)$ be a constrained \Ustat{} defined as above, 
with $(X_i)\xoo$ i.i.d.,
and assume
  $\E|f(X_1,\dots,X_D)|^2<\infty$.
Then, with $\mu=\mu_f$ given by \eqref{UUmu} and some $\gss=\gss_f\ge0$,
\begin{align}\label{puu1}
  \frac{\UU_n-\binom{n}{d}\mu}{n^{d-1/2}}
\dto \N\bigpar{0,\gss}
.\end{align}

Moreover, if $\E|f(X_1,\dots,X_d)|^p<\infty$ for some $p\ge2$, 
the \eqref{puu1} holds with convergence of all moments of order $\le p$.
\end{prop}

It does not matter whether we subtract $\E\UU_n$ 
or $\binom nd \mu$ in
\eqref{puu1}, since the difference is
$O\bigpar{n^{d-1}}=o\bigpar{n^{d-1/2}}$
by \eqref{EUU}.

\begin{proof}
This is a special case of \cite[Theorems 3.9 and 3.15]{SJ360}. 
\end{proof}

The variance $\gss$ in \eqref{puu1} can be calculated explicitly, see
\cite[Remark 6.2]{SJ360}, but the formulas are a bit complicated, and we 
omit them. Instead, we give a  criterion that often can be used in applications
to show that $\gss>0$.
We define, in analogy with \eqref{fi},
\begin{align}\label{fj}
  f_j\bigpar{x_1,\dots,x_{\mb_j}}
:=
\E\bigsqpar{ f\bigpar{X_1,\dots,X_D} 
\mid (X_{D_{j-1}+1},\dots,X_{D_{j}}) = (x_1,\dots,x_{\mb_j})}
.\end{align}
We extend the definition \eqref{Snh} to functions $g:\cS^b\to\bbR$ for any
$b\ge1$ by defining, for such $g$,
\begin{align}\label{emma}
S_n(g)&:=\sum_{i=1}^{n} g\bigpar{X_{i},\dots,X_{i+b-1}}
.\end{align}

\begin{prop}\label{PUU0}
  In \refP{PUU1}, the asymptotic variance $\gss_f=0$ if and only if
for every $j\in[d]$, there exists a function $\psi_j:\cS^{\mb_j-1}\to\bbR$ such
that
\as{}
\begin{align}\label{fpsi}
  f_j\bigpar{X_1,\dots,X_{b_j}}-\mu
=\psi_j\bigpar{X_2,\dots,X_{b_j}}
-\psi_j \bigpar{X_1,\dots,X_{b_j-1}},
\end{align}
and thus \as, for every $n\ge 1$,
\begin{align}\label{williams}
S_n(f_j-\mu)
=\psi_j\bigpar{X_{n+1},\dots,X_{n+b_j-1}}
-\psi_j \bigpar{X_1,\dots,X_{b_j-1}}.
\end{align}
Consequently, if $\gss_f=0$, then
 $S_n(f_j)$ is independent of $X_{b_j},\dots,X_n$ for every $j\in[d]$ and
 $n\ge b_j$.
\end{prop}
\begin{proof}
  This is essentially a special case of \cite[Theorem 8.4]{SJ360};
the difference is mainly notational.
The function $g_j$ in \cite[Theorem 8.4 and Remark 6.2]{SJ360} is, in our
case, given by 
\begin{align}\label{gjfj}
  g_j\bigpar{x_1,\dots,x_{\MDd+1}}
= f_j\bigpar{x_{\MD_{j-1}+1}, \dots,x_{\MD_{j-1}+b_j}}-\mu;
\end{align}
thus $g_j$ is essentially the same as $f_j-\mu$ but contains some redundant
variables. 
\cite[Theorem 8.4]{SJ360} says that $\gss_f=0$ if and only if there exists a
function $\gf_j:\cS^{\MD_d}\to\bbR$ such that a.s.
\begin{align}\label{gphi}
  g_j\bigpar{X_1,\dots,X_{\MD_{d}+1}}
=\gf_j\bigpar{X_2,\dots,X_{\MD_{d}+1}}
-\gf_j \bigpar{X_1,\dots,X_{\MD_{d}}}.
\end{align}
This is \eqref{fpsi}, except that we have  redundant variables.
These may be eliminated one by one. For example, if $\MD_{j-1}>0$, and thus
$g_j$ does not depend on $x_1$ by \eqref{gjfj}, 
then \eqref{gphi} implies that 
for \aex{} fixed $x_1\in \cS$, we have
$\gf_j \bigpar{X_1,\dots,X_{\MDd}} = \gf_j \bigpar{x_1,X_2,\dots,X_{\MDd}}$
a.s., and thus a.s.
\begin{align}
  \gf_j \bigpar{X_1,\dots,X_{\MDd}} = \gf'_j \bigpar{X_2,\dots,X_{\MDd}}
\end{align}
for some function $\gf'_j:\cS^{\MD_d-1}\to\bbR$.
Continuing in this way, from both ends, we see that a.s.
\begin{align}\label{phipsi}
 \gf_j\bigpar{X_1,\dots,X_{\MDd}}
= \psi_j\bigpar{X_{\MD_{j-1}+1}, \dots,X_{\MD_{j-1}+b_j-1}}
\end{align}
for some function $\psi_j$, and thus 
\eqref{gphi} reduces to \eqref{fpsi}.
(Alternatively, one might note that \eqref{gphi} implies 
$\Var\bigsqpar{S_n(f_j-\mu)}=\Var S_n(g_j)=O(1)$, and then 
\cite[Theorem 2]{SJ286} 
yields \eqref{fpsi} -- this essentially repeats part of the argument 
in \cite{SJ360} yielding \eqref{gphi}.)
Conversely, \eqref{fpsi} trivially yields \eqref{gphi} for a suitable $\gf_j$.
\end{proof}

We will use a renewal theory version of constrained \Ustat{s}.
We assume again that $h:\cS\to\bbR$ with $h(X_i)\ge0$ a.s., and use
the notation \eqref{nu}--\eqref{Nx}.
The following results are special cases of 
\cite[Theorems 3.20, 8.7, 3.21, and 3.23]{SJ360}.

\begin{prop}[\cite{SJ360}]\label{PUUN}
  Let $\UU_n=\UU_n(f)$ be a constrained \Ustat{} defined as above, 
with $(X_i)\xoo$ i.i.d.
Suppose that 
$\E|f(X_1,\dots,X_D)|^2<\infty$, 
and that
$h(X)\ge0$ \as, with $\mux:=\E h(X)>0$
and $\E h(X)^2<\infty$.
Then, with notations as above, 
as \xtoo,
\begin{equation}\label{cvtauu}
\frac{\UU_{N(x)}-\mu{\mux}^{-d}{d!}\qw x^{d}}
{x^{d-1/2}} 
\dto \N\bigpar{0,\gamx^2},
\end{equation}
for some $\gamxx\ge0$.
Moreover,
$\gamxx>0$ unless,
for each $j=1,\dots,d$, 
the conditions in \refP{PUU0} hold
with $f-\mu$ replaced by the function
$f_j(X_1,\dots,X_{b_j})-\frac{\mu}{\nu} h(X_1)$.
\end{prop}

\begin{proof}
  The limit \eqref{cvtauu} is a special case of \cite[Theorem 3.20]{SJ360}.
The only detail that requires a comment is that \cite[Theorem 8.7]{SJ360}
says that if $\gamm=0$, then a.s.
\begin{align}\label{kvinna}
  g_j\bigpar{X_1,\dots,X_{\MD_{d}+1}}+\mu
-\frac{\mu}{\nu}h(X_1)
=\gf_j\bigpar{X_2,\dots,X_{\MD_{d}+1}}
-\gf_j \bigpar{X_1,\dots,X_{\MD_{d}}}
\end{align}
for some function $\gf$, where as above $g_j$ is given by \eqref{gjfj}.
If we use \eqref{gjfj} and  define
\begin{align}\label{qv}
  \bgf_j\bigpar{x_1,\dots,x_{\MD_d}}:=
  \gf_j\bigpar{x_1,\dots,x_{\MD_d}}
-\sum_{i=1}^{\MD_{j-1}}\frac{\mu}{\nu}h(x_i), 
\end{align}
then \eqref{kvinna} is equivalent to
\begin{align}\label{kv2}
  f_j\bigpar{X_{\MD_{j-1}+1},\dots,X_{\MD_{j-1}+b_j}}
-\frac{\mu}{\nu}h\bigpar{X_{\MD_{j-1}+1}}
=\bgf_j\bigpar{X_2,\dots,X_{\MD_{d}+1}}
-\bgf_j \bigpar{X_1,\dots,X_{\MD_{d}}}.
\end{align}
The result follows by eliminating redundant variables as 
in the proof of \refP{PUU0}.
\end{proof}

\begin{prop}[\cite{SJ360}] \label{PUUN2}
Suppose in addition to the hypotheses in \refP{PUUN} that $h(X)$ is
integer-valued.
Then \eqref{cvtauu} holds also conditioned on $S_{N(x)}=x$ 
for integers $x\to\infty$.
\qed
\end{prop}


\begin{prop}[\cite{SJ360}]\label{PUUNmom}
Suppose in addition to the hypotheses in \refP{PUUN} or \ref{PUUN2} 
that 
$\E|f(X_1,\dots,X_D)|^p<\infty$ and  $\E|h(X)|^p<\infty$ for every $p<\infty$.
Then the conclusion \eqref{cvtauu} holds 
with convergence of all moments.
\qed
\end{prop}


\begin{remark}\label{Rmulti2}
Again, the results in \refPs{PUU1} and \ref{PUUN}--\ref{PUUNmom} 
hold jointly for several $f$
(possibly with different $d$ and $b_1,\dots,b_d$), see \cite{SJ360}.
\end{remark}

\section{Patterns and codes of tree permutations}
\label{Scode}

Consider an occurrence of a tree permutation $\gs\in\fT_\ell$ in another tree
permutation $\tau\in\fT_n$.
The occurrence is defined by a subset $I=\set{i_1,\dots,i_\ell}$ of the index
set $[n]$.
We colour each symbol in the code $\gO_\tau$ \emph{red} if its index belongs to
$I$, and \emph{black} otherwise. We use also the same colours for the
corresponding vertices in $G_\tau$.
(All colourings in this paper are in red and black. We may regard the red
symbols or vertices as marked.)

Note that in the resulting coloured copy of $\gO_\tau$,
the red symbols form the code $\gO_\gs$ of $\gs$; this is a consequence of
\eqref{L}--\eqref{R} and 
the fact that the corresponding (red) induced subgraph of $G_\tau$ 
equals $G_\gs$ up to an order-preserving relabelling.
However, not every subset of $\ell$ symbols in the right order
corresponds to an occurrence of $\gs$.
There is a 1--1 correspondence between 
\begin{enumerate}
\item (nonempty) subsets of $[n]$,  
\item (nonempty) subsequences of $\gO_\tau$,
\item occurences of some permutation $\gu$ in $\tau$,
\item (nonempty) labelled subgraphs of the permutation graph $G_\tau$.
\end{enumerate}
However, the subgraph in (4) is not necessarily a tree,
and thus,
the permutation $\gu$ in (3) is not necessarily a tree permutation.

We may characterize the subsets of symbols in $\gO_\tau$ that yield occurences
of $\gs$ as follows.

\begin{lemma}\label{LA}
Let $\tau$ and $\gs$ be  tree permutations with $|\tau|\ge|\gs|\ge2$.
A colouring of the code $\gO_\tau$
corresponds to an occurrence of $\gs$ in $\tau$
if and only if we may the delete the black symbols one by one 
in some order
according to the following rules 
(always interpreted for the current string)
until only red symbols remain, and these form the code $\gO_\gs$. 
The allowed deletions are
(in any order, and possibly repeated):
\begin{PXenumerate}{A}
\item\label{LA1} 
a black $\sfL$ that is immediately followed by another $\sfL$;
\item\label{LA2} 
a black $\sfL$ in the last but one position; 
\item\label{LA3} 
a black $\sfR$ that is immediately preceded by another $\sfR$;
\item\label{LA4} a black $\sfR$ in position $2$.
\end{PXenumerate}
\end{lemma}

\begin{proof}
Consider first the case of deleting one
vertex $i\in[n]$ from the tree $G_\tau$, \ie, restricting the permutation 
$\tau$ to $[n]\setminus\set i$ and then relabelling to get
a permutation $\tau_1$ in $\fS_{n-1}$. The permutation graph $G_{\tau_1}$ is
an induced subgraph of $G_\tau$, and is thus always a forest; it is a tree
if and only if it is connected, which is the case exactly when $i$ is leaf
in $G_\tau$. 
By \refL{Lleaf}, the black vertices that may be deleted 
leaving a tree correspond precisely to the symbols listed in 
\ref{LA1}--\ref{LA4}.

Thus, to repeatedly remove black symbols according to the rules in the lemma,
is equivalent to repeatedly removing black leaves of $G_\tau$, 
leaving a red subtree; 
if the resulting red code is $\gO_\gs$, then this yields an occurence of $\gs$.

Conversely, if the colouring of $\gO_\tau$ corresponds to an occurrence of
$\gs$ in $\tau$, 
then the red vertices form a red subtree in $G_\tau$, and we may
remove the black vertices of $G_\tau$ is some order such that we always
remove a black leaf of the current tree; this means that we may remove the
black symbols in some order such that the rules 
\ref{LA1}--\ref{LA4} are followed.
\end{proof}

We may invert the deletions in \refL{LA}, and instead insert black symbols
into $\gO_\gs$.

\begin{lemma}\label{LB}
Let $\tau$ and $\gs$ be  tree permutations with $|\tau|\ge|\gs|\ge2$.
A colouring of the code $\gO_\tau$
corresponds to an occurrence of $\gs$ in $\tau$
if and only if we may obtain it by 
from a red code $\gO_\gs$ by inserting black symbols one by one 
according to the following rules 
(always interpreted for the current string).
The allowed insertions are
(in any order, and possibly repeated):
\begin{PXenumerate}B
\item\label{LB1} a black $\sfL$ immediately to the left of any $\sfL$;
\item\label{LB2} a black $\sfL$ immediately to the left of the last symbol;
\item\label{LB3} a black $\sfR$ immediately to the right of any $\sfR$;
\item\label{LB4} a black $\sfR$ immediately to the right of the first symbol.
\end{PXenumerate}
\end{lemma}

\begin{proof}
  Immediate from \refL{LA}.
\end{proof}

We have so far considered deleting or inserting one symbol at a time. 
Since only the end result matters, the following version is 
more convenient  for our purposes.
(Recall that the first red symbol always is $\sfL$, and the last is $\sfR$.)

\begin{lemma}\label{LC}
Let $\gs$ be a  tree permutation with $|\gs|\ge2$.
A coloured code $\gO$
corresponds to a marked (red) occurrence of $\gs$ in some tree permutation
$\tau$
if and only if we may obtain $\gO$ 
from a red code $\gO_\gs$ by inserting black symbols 
as follows
(the strings may be empty):
\begin{PXenumerate}C
\item\label{LC1} 
a string of black $\sfL$ immediately to the left of each red $\sfL$
except the first;
\item\label{LC2} 
a string of black $\sfL$ immediately to the left of the last red $\sfR$;
\item\label{LC3} 
a string of black $\sfR$ immediately to the right of each red $\sfR$
except the last;
\item\label{LC4} 
a string of black $\sfR$ immediately to the right of the first red $\sfL$;
\item\label{LC5} 
any black string that is empty or begins with $\sfL$ before the first
red  symbol; 
\item\label{LC6} 
any black string that is empty or ends with $\sfR$ after the last red
  symbol.
\end{PXenumerate}
\end{lemma}

\begin{proof}
It is easily seen that if we take any coloured code obtained by these
rules, and insert another black symbol according to the rules in \refL{LB},
then the result is also described by \ref{LC1}--\ref{LC6}.
Hence, by induction, all possible coloured codes are given by the insertions
\ref{LC1}--\ref{LC6}.

Conversely, suppose that $\gO$ is obtained from a red $\gO_\gs$ by 
\ref{LC1}--\ref{LC6}; we have to show that
it also can be obtained by repeating
\ref{LB1}--\ref{LB4} in some order.
Evidently, 
\ref{LC1}--\ref{LC4} can be obtained by repeating
\ref{LB1}--\ref{LB4}, so it remains only to show that
we may add an arbitrary black string beginning with $\sfL$
before the first red symbol, and an arbitrary black string ending with $\sfR$ 
after the last red symbol. 
To see this, note that we may first add a black
$\sfL$ to the left by \ref{LB1}. 
Then, when the code begins with a black $\sfL$,  
we may by  either add a black $\sfR$ as the second symbol by \ref{LB4},
or a black $\sfL$ as the first symbol by \ref{LB1}, but the latter
gives the same result as adding a black $\sfL$ as the second symbol. 
Hence, we may add
an arbitrary black symbol immediately after the first one, and by repeating
this we may obtain any black string beginning with $\sfL$, 
verifying \ref{LC5}. 
The argument for the right side is symmetric.
\end{proof}

\begin{lemma}\label{LAz}
Fix a tree permutation $\gs$ with $|\gs|\ge2$. 
For every $n$, let $a_{n;\gs}$   
be the 
number of pairs $(\tau,\gs')$ of a tree permutation $\tau$ of length $|\tau|=n$
together with a marked occurence $\gs'$ of the pattern $\gs$.
Define also the generating function
\begin{align}\label{Az}
A_{\gs}(z):=\sum_{n\ge|\gs|} a_{n;\gs}z^n.
\end{align}
Then,
\begin{align}\label{Az1}
  A_\gs(z)=\frac{z^{|\gs|}}{(1-z)^{|\gs|-2}(1-2z)^{2}}.
\end{align}

\end{lemma}

\begin{proof}
By \refL{LC}, $a_{n;\gs}$ equals the number of coloured codes of length $n$
that can be obtained 
from a red $\gO_\gs$ by the rules \ref{LC1}--\ref{LC6}.  
These insertions are independent of each other, so they correspond to
multiplying factors in the generating function $A_\gs(z)$.

Each possible application of \ref{LC1}--\ref{LC4} yields a factor
$\sumko z^k=(1-z)\qw$. There is one possible such application for each
symbol in $\gO_\gs$, by \ref{LC1} or \ref{LC4} for each $\sfL$, and
by \ref{LC2} or \ref{LC3} for each $\sfR$. Hence, the total contribution of 
\ref{LC1}--\ref{LC4} is 
$(1-z)^{-|\gs|}$.

By \ref{LC5}, we may to the left add a black prefix that is either empty or
is an
arbitrary sting beginning with $\sfL$, which gives $2^{k-1}$ possible
prefixes of length $k$ for every $k\ge1$
(and 1 prefix of length 0). This contributes to $A_\gs(z)$ a
factor
\begin{align}
  1+\sumk 2^{k-1}z^k = 1+\frac{z}{1-2z} = \frac{1-z}{1-2z}.
\end{align}
Black suffixes by \ref{LC6}  contribute the same factor.
These factors all multiply the term corresponding to the original red symbols
$\gO_\gs$, which is $z^{|\gs|}$. Hence, we obtain
\begin{align}\label{Az2}
  A_\gs(z)=z^{|\gs|}(1-z)^{-|\gs|}\Bigpar{\frac{1-z}{1-2z}}^2,
\end{align}
which yields \eqref{Az1}.
\end{proof}

This yields an exact formula for the expected number of occurences of $\gs$;
note that the result depends only on $|\gs|$ and $n$.

\begin{theorem}
  \label{TE}
Fix a tree permutation $\gs$ with $|\gs|\ge2$. Then, for $n\ge|\gs|$,
\begin{align}\label{te}
  \E\ns(\btau_n) &
= [z^n] \bigpar{z^{|\gs|}(2-z)^{2-|\gs|}(1-z)^{-2}}
= [z^{n-|\gs|}] \bigpar{(2-z)^{2-|\gs|}(1-z)^{-2}}
\notag\\&
=n+3-2|\gs|+2^{-n}\sum_{i=0}^{|\gs|-3}(|\gs|-2-i)2^{|\gs|-i-1}\binom{n-|\gs|+i}{i}
.\end{align}
\end{theorem}

\begin{proof}
  The total number of occurences of $\gs$ in tree permutations of length $n$
  is $a_{n;\gs}$, and the number of such tree permutations is $t_n=2^{n-2}$
  by \eqref{tn}.
Hence, by \eqref{Az}--\eqref{Az1},
\begin{align}\label{te1}
    \E\ns(\btau_n) &
= \frac{a_{n;\gs}}{2^{n-2}}
=[z^n] \bigpar{2^{2-n}A_n(z)}
=[z^n]\bigpar{4 A_n(z/2)}
\notag\\&
=[z^n]\frac{z^{|\gs|}}{(2-z)^{|\gs|-2}(1-z)^2},
\end{align}
which gives the first two expressions in \eqref{te};
the explicit formula then follows from the partial fraction expansion,
with $m=|\gs|-2\ge0$,
\begin{align}\label{te2}
  \frac{1}{(2-z)^{m}(1-z)^2}
=\frac1{(1-z)^2}-\frac{m}{1-z}+\sum_{j=1}^{m}\frac{m-j+1}{(2-z)^j}
.\end{align}
\end{proof}

\section{A random tree permutation of random length}
\label{SRT}

Recall that $T(z)$ is the generating function in \eqref{Tz}, and
let, throughout the paper, $p$ be the (unique) positive root of 
\begin{align}\label{Tp}
  T(p)=1.
\end{align}
By \eqref{Tz}, this
yields $0<p<1/2$ and 
$p-p^2=1-2p$, or
$p^2-3p+1=0$, and thus 
\begin{align}\label{p}
  p=\frac{3-\sqrt{5}}2
=0.381966\dots   
.
\end{align}
Recalling the golden ration $\phi$ in \eqref{phi},
we thus have 
\begin{align}\label{p2phi}
p=\phi\qww=2-\phi.
\end{align}
We note also
\begin{align}\label{pphi+}
  1-p&=\phi-1=\phi\qw,
&
1-2p&=p(1-p)=\phi^{-3}.
\end{align}

We now define a random tree permutation $\tbtau$ to be a 
random element of $\fT_*$ with the distribution
\begin{align}\label{tbtau}
  \P(\tbtau=\tau)=p^{|\tau|},
\qquad \tau\in\fT_*.
\end{align}
Note that the sum over all $\tau\in\fT_*$ of the probabilities in
\eqref{tbtau} equals $\sum_n t_np^n=T(p)=1$,
and thus \eqref{tbtau} really defines a probability distribution.

The random tree permutation  $\tbtau$ thus has random length.
It follows from \eqref{tbtau} that the probability generating function of
  $|\tbtau|$ is
  \begin{align}\label{Gtau}
\Gtau(z):=\sumn t_np^nz^n = T(pz).
  \end{align}

\begin{lemma}\label{Ltau}
  We have
  \begin{align}
    \E|\tbtau|&
=\phi+2
=\frac{5+\sqrt5}{2}
=\sqrt5\,\phi
\doteq 3.618 
\label{Etau},
\\
  \E|\tbtau|^2&
=11\phi+8=\frac{27+11\sqrt5}2\doteq 25.798 
\label{Etau2},
\\
\Var|\tbtau|&=6\phi+3 
= 3\phi^3
=6 + 3\sqrt5 \doteq 12.708, 
\label{Vtau}\\
  \E|\tbtau|^k&<\infty, \qquad \forall k<\infty.
\label{Etauk}  \end{align}
\end{lemma}

\begin{proof}
By \eqref{Gtau} and straightforward calculations
using \eqref{p2phi}--\eqref{pphi+},
\begin{align}\label{et1}
  \E|\tbtau|=\Gtau'(1)=pT'(p)=\frac{p(1-2p+2p^2)}{(1-2p)^2}
=\phi^4(\phi^{-3}+2\phi^{-4})
=\phi+2
.\end{align}
Similarly,
\begin{align}\label{et2}
  \E\bigsqpar{|\tbtau|(|\tbtau|-1)}
=\Gtau''(1)=p^2T''(p) =\frac{2p^2}{(1-2p)^3}=2\phi^5
=10\phi+6
\end{align}
and thus, combining \eqref{et1} and \eqref{et2},
\begin{align}\label{et3}
  \E|\tbtau|^2
=11\phi+8
\end{align}
and
\begin{align}\label{et4}
  \Var|\tbtau|=\bigpar{11\phi+8}-(\phi+2)^2 = 6\phi+3
.\end{align}
This shows \eqref{Etau}--\eqref{Vtau}.

Finally, \eqref{Etauk} follows because $\Gtau(z)$ has radius of convergence
greater than 1. (Or directly from \eqref{tn} and \eqref{tbtau}.)
\end{proof}

\subsection{From random trees to random forests}\label{SSRT}
Recall that forest permutations are  sums of tree permutations
\eqref{tomtam}.
Let $\tbtau_1,\tbtau_2,\dots$ be an infinite sequence of independent random
tree permutations with the distribution \eqref{tbtau}, and
let 
\begin{align}
  S_m:=\sum_{i=1}^m |\tbtau_i|,
\end{align}
the total length of the $m$ first of these tree permutations.
Thus, for any $m\ge1$, $\tbtautaum$ is a forest permutation of length $S_m$,
having $m$ blocks.

Suppose that $\pi$ is a forest permutation with $m$ blocks
$\tau_1,\dots,\tau_m$. Then, by \eqref{tbtau},
\begin{align}\label{p1}
  \P\bigpar{\tbtautaum=\pi}
=\P\bigpar{\tbtau_i=\tau_i,\forall i\le m}
=\prod_{i=1}^m\P\bigpar{\tbtau_i=\tau_i}
=\prod_{i=1}^m p^{|\tau_i|}
=p^{|\pi|}.
\end{align}
Note that this depends only on $|\pi|$.

In order to obtain arbitrary forest permutations, we have to consider a random
number of blocks. We use a renewal theoretic approach.
For any $n\ge1$, 
let, as in \eqref{Nx},
\begin{align}\label{Nn}
  N(n):=\min \set{m\ge1: S_m\ge n}.
\end{align}
Then, $S_{N(n)}\ge n$. 
Moreover, if $\pi\in\fF_n$ has $m$ blocks $\pi_1,\dots,\pi_m$, then
$\tbtautaum=\pi$ entails $S_m=|\pi|=n$, and thus
$N(n)=m$.
Hence, using also \eqref{p1},
\begin{align}\label{piN}
  \P\bigpar{\tbtautauN=\pi}
&=\P\bigpar{N(n)=m\;\&\; \tbtautaum=\pi}
\notag\\&
=\P\bigpar{\tbtautaum=\pi}
=p^{|\pi|}=p^n.
\end{align}
This probability is thus the same for all $\pi\in\fF_n$.
Consequently, conditioned on $S_{N(n)}=n$, so that
$\tbtautauN\in\fF_n$, \eqref{piN} implies that
$\tbtautauN$ has the uniform distribution in $\fF_n$, and thus
\begin{align}\label{pint}
  \bpi_n \eqd \bigpar{\tbtautauN\mid S_{N(n)}=n}.
\end{align}
In words, we can construct a uniformly random $\bpi_n\in\fF_n$
from the infinite sequence $(\tbtau_i)$ by composing $\tbtau_1,\tbtau_2,\dots$ 
until their total length is at least $n$, and then condition on the
total length being exactly $n$.

\section{Trees in a random tree permutation
$\tbtau$}

%

The construction \eqref{pint} suggests that it is useful to study the random
variable 
$\ns(\tbtau)$, for a given  permutation $\gs$.
We do this first for a tree permutation
$\gs$. 

\begin{lemma}\label{L8}
  Let $\gs$ be a tree permutation, 
and let $\tbtau$ be random with the distribution \eqref{tbtau}.
Then,
\begin{align}
  &\mu_\gs:=\E[\ns(\tbtau)]
=
    \begin{cases}
\E|\tbtau|=\phi+2,
&|\gs|=1
,\\
      p^{|\gs|}(1-p)^{-|\gs|}\bigpar{\frac{1-p}{1-2p}}^2
= 
p^{|\gs|/2-2}
=\phi^{4-|\gs|},
&|\gs|\ge2
.    \end{cases}
\label{l81}
\\&\E[\ns(\tbtau)^k] <\infty, \qquad \forall k\ge1\label{l8k}
.\end{align}
\end{lemma}

\begin{proof}
First, if $|\gs|=1$, i.e., $\gs=1$, 
then trivially $\occ_\gs(\tau)=|\tau|$ for any permutation $\tau$, 
and thus this case of \eqref{l81} follows from \refL{Ltau}.

Assume now $|\gs|\ge2$, and 
let $a_{n;\gs}$ and $A_\gs(z)$ be as in \refL{LAz}.
Then, 
\begin{align}
\sum_{\tau\in\fT_n}\ns(\tau)=a_{n;\gs},
\end{align}
and thus it follows from \eqref{tbtau} that
\begin{align}\label{Ap}
  \E \ns(\tbtau) =
\sum_{\tau\in\fT_*} \ns(\tau) p^{|\tau|}
=\sum_{n\ge|\gs|} p^n \sum_{\tau\in\fT_n}\ns(\tau)
=\sum_{n\ge|\gs|} p^na_{n;\gs}
=A_\gs(p).
\end{align}
Consequently, \eqref{l81} follows from \eqref{Ap} och \eqref{Az2},
using 
\eqref{pphi+}.

Finally, \eqref{l8k} follows from \eqref{Etauk}, since $\ns(\tau)\le|\tau|$
for any $\gs$. 
\end{proof}
  
\begin{example}\label{E21}
  The only tree permutation $\gs$ with $|\gs|=2$ is 21, and $\occ_{21}(\tau)$
  counts the number of inversions in $\tau$, i.e., the number of edges in
  $G_\tau$. If $\tau$ is a tree permutation, we thus have
  $\occ_{21}(\tau)=|\tau|-1$. Indeed, \refL{L8} yields
$\E\occ_{21}(\tbtau)=\phi^2$, which equals  
$\E\bigsqpar{|\tbtau|-1}=\E|\tbtau|-1=\phi+1$
given by \refL{Ltau}.
\end{example}

\section{Patterns in a random forest permutation}
\label{SpfF}

We are now prepared to prove \refTs{TF} and \ref{TFI} on 
patterns in $\pi_n$.

\begin{proof}[Proof of \refT{TF}]
  Let $\pi\in\fS_n$ have block decomposition
$\pi=\pi_1\opluss\pi_\nuN$.
If $\gs=\gs_1\opluss\gs_d$ occurs as a pattern in $\pi$, then each block
$\gs_j$ 
is mapped into some block $\pi_{i_j}$, but it is possible that several
blocks of $\gs$ fit in the same block of $\pi$.
Let $\occ'_\gs(\pi)$ be the number of occurrences of $\gs$ such that the blocks
are mapped to different blocks in $\pi$, i.e., where the function $j\mapsto
i_j$ is injective, and let $\occ''_\gs(\pi)$ denote the number
of the remaining occurrences.

Let us first consider $\occ'_\gs$, which will be the main term.
We have
\begin{align}
  \occ'_\gs(\pi)=\sum_{1\le i_1<\dots<i_d\le \nuN}\prod_{j=1}^d\occ_{\gs_j}(\pi_{i_j}).
\end{align}
Thus, by \eqref{pint},
\begin{align}
  \occ'_\gs(\bpi_n)
\eqd\Bigpar{\sum_{1\le i_1<\dots<i_d\le N(n)}\prod_{j=1}^d\occ_{\gs_j}(\tbtau_{i_j})
\Bigm| S_{N(n)}=n}
.\end{align}
This is a conditioned \Ustat{} of the type in \refP{PUN2},
based on the \iid{} sequence $X_i:=\tbtau_i$, with 
$\cS=\fSx$, the (discrete) space of all permutations, and
$h(\tau):=|\tau|$;
more precisely,
we then have
$\occ'_\gs(\bpi_n)\eqd
\bigpar{U_{N(n)}(f)\mid S_{N(n)}=n}$
with 
\begin{align}\label{f10}
f\bigpar{\tau_1,\dots,\tau_d}:=  \prod_{j=1}^d\occ_{\gs_j}(\tau_{j}).
\end{align}
Note that \eqref{l8k} and \Holder's inequality imply that
$\E \bigsqpar{\bigabs{f\xpar{\tbtau_1,\dots,\tbtau_d}}^p}<\infty$ for
every $p<\infty$.
Similarly, $\E\bigsqpar{h(\tbtau_1)^p}<\infty$ by \eqref{Etauk}.

It follows from \refP{PUN2}
that \eqref{tf1} holds for $\occ'_\gs$,
with 
some $\tmu_\gs$ and $\gamm_\gs$;
note that in the notation of \refS{SSU}, 
by \refL{Ltau},
\begin{align}\label{fnu}
  \nu:=\E h(\tbtau)=\E|\tbtau|=\phi+2
,\end{align}
and
by \eqref{Umu}, \eqref{f10}, the independence of $\tbtau_i$, 
and \eqref{l81} in \refL{L8},
\begin{align}\label{mugs}
\mu=
\mu_\gs&:=\prod_{j=1}^d\E\bigsqpar{\occ_{\gs_j}(\tbtau_{j})}
=\prod_{j=1}^d\mu_{\gs_j}
=(\phi +2)^{\gl}\phi^{4(d-\gl)-(|\gs|-\gl)}
.\end{align}
 Thus, by \eqref{cvtau}, $\tmu_\gs$ in \eqref{tf1} 
(so far for $\ns'$)
is given by
\begin{align}\label{tmugs2}
\tmu_{\gs}&
=\frac{\mu_\gs}{ \nu^{d}d!}
=\frac{\mu_{\gs}}{(\phi+2)^{d}d!}
=\frac{1}{d!}(\phi+2)^{\gl-d}\phi^{4d-3\gl-|\gs|}
,\end{align}
which yields \eqref{tmugs}.

Similarly, by \eqref{fi},
\begin{align}\label{f12}
  f_i(\tau)&= \occ_{\gs_i}(\tau)\prod_{j\neq i}\E\occ_{\gs_j}(\tbtau_{j})
=\prod_{j\neq i}\mu_{\gs_j} \cdot \occ_{\gs_i}(\tau)
=\frac{\mu_\gs}{\mu_{\gs_i}}\occ_{\gs_i}(\tau).
\end{align}
Suppose that $|\gs_i|>1$.
We may have, with positive probabilities, 
\begin{enumerate}
\item 
 $|\tbtau|=1$, and then 
$\occ_{\gs_i}(\tbtau)=0$, 
\item $\tbtau=\gs_i$, and then $\occ_{\gs_i}(\tbtau)=1>0$.
\end{enumerate}
Thus it is impossible to have
$f_i(\tbtau)=c|\tbtau|$ \as, for any real $c$.
Consequently, \refP{PUN} yields 
 $\gamm_\gs>0$ if any block $\gs_i$ with $|\gs_i|>1$ exists.

It remains to show that $\occ''_{\gs}(\bpi_n)$ is negligible.
By grouping the blocks of $\gs$ that are mapped into the same block of
$\pi$,
we see that $\occ''_{\gs}(\pi)$ can be written as a sum over all
decompositions $\gs=\tgs_1\opluss\tgs_k$ with $k<d$,
of the number of occurrences with each $\tgs_i$ mapped into a block of
$\pi$, with these blocks distinct.
(Here $\tgs_i$ are
necessarily forest permutations.)
It follows, using again \eqref{pint}, and $N(n)\le n$, that
\begin{align}
  \E\occ''_{\gs}(\bpi_n)
&\le \frac{1}{\P(S_{N(n)=n})}\E
  \occ''_{\gs}\bigpar{\tbtau_1\opluss\tbtau_{N(n)}}
\le C\E \occ''_{\gs}\bigpar{\tbtau_1\opluss\tbtau_{n}}
\notag\\&
=C\sum_{k=1}^{d-1}\sum_{\tgs_1,\dots,\tgs_k}\sum_{1\le i_1<\dots<i_k\le n}
\E\prod_{j=1}^{k} \occ_{\tgs_j}(\tbtau_{i_j}).
\end{align}
The number of terms in the  multiple sum is $O\bigpar{n^{d-1}}$, and
each term is $O(1)$, using independence, 
the trivial  $\occ_{\tgs_j}(\tbtau)\le |\tbtau|^{|\gs_j|}$,
 and \eqref{Etauk}.
Hence, $\E\occ''_{\gs}(\bpi_n)=O\bigpar{n^{d-1}}$,
and \eqref{tf1} follows from the result for $\occ'_\gs(\bpi_n)$.

Moment convergence follows in the same way,
using \refP{PUNmom} and Minkowski's inequality; we omit the details.
\end{proof}

\begin{proof}[Proof of \refT{TFI}]
We have the trivial identity
\begin{align}\label{triv}
  \sum_{\gs\in\fS_d}\occ_{\gs}(\bpi_n)=\binom nd
.\end{align}
Furthermore, we only have to consider forest permutations $\gs\in\fF_d$ in
\eqref{triv}, since otherwise $\occ_{\gs}(\bpi_n)=0$.

Let $\gs\in\fF_d$, and let $d'$ be its number of blocks.
If $\gs\neq\iota_d$, then $d'<d$. 
If $d'\le d-2$, then \eqref{tf1} implies that
$\occ_{\gs}(\bpi_n)/n^{d-3/2}\pto0$, so such terms can be ignored.

The remaining terms in \eqref{triv} have $d'=d-1$, and thus 1 block of
length 2 and $d-2$ blocks of length 1. There are $d-1$ such
permutations; for example, if $d=4$, they are 2134, 1324 and 1243.
For each such $\gs$, we have by
\eqref{tmugs} 
\begin{align}
  \tmu_\gs=\frac{1}{(d-1)!}(\phi+2)\qw\phi^{4(d-1)-3(d-2)-d}
=\frac{1}{(d-1)!}(\phi+2)\qw\phi^{2}
,\end{align}
where,
see \eqref{Etau}, 
\begin{align}\label{ellen}
(\phi+2)\qw \phi^2=\frac{\phi^2}{\sqrt 5 \phi}
=\frac{\phi}{\sqrt 5}=\frac{5+\sqrt5}{10}.
\end{align}
Hence, 
\refT{TF} yields
\begin{align}\label{cf5}
  \frac{ \occ_{\gs}(\bpi_n)-\tmuinv(d-1)!\qw n^{d-1}}{n^{d-3/2}}
\dto \N\bigpar{0,\gamm_{\gs}},
\end{align}
Moreover, the proof of \refT{TF} applies also to the
sum $\sumxi_{\gs} \occ_{\gs}$
over these $d-1$ permutations  $\gs$.
(Consider the sum of the corresponding functions \eqref{f10}. 
See also \refR{RCW}.)
Thus,
\begin{align}\label{cf7}
 \frac{\sumxi_{\gs} \occ_{\gs}(\bpi_n)-\tmuinv(d-2)!\qw n^{d-1}}
{n^{d-3/2}}
\dto \N\bigpar{0,\gamm},
\end{align}
where $\gamm>0$ by the argument in the proof of \refT{TF}.

As said above, we may add all $\gs\in\fF_k$ with less than $d-1$ blocks 
to the sum in \eqref{cf7} without changing the limit.
The resulting sum is, by \eqref{triv},
\begin{align}
  \sum_{\gs\in\fF_d\setminus\set{\iota_d}}\occ_{\gs}(\bpi_n)
=\binom{n}{d} -\occ_{\iota_d}(\bpi_n),
\end{align}
and thus \eqref{cf1} follows, with $\gamm_{\iota_d}=\gamm$ in \eqref{cf7},

Moment convergence follows by the same argument.
\end{proof}

\begin{remark}\label{RFVar}
  The asymptotic variance $\gamm_\gs$ can by \eqref{cvtau2} and \eqref{Ugss}
  be computed from variances and covariances of the
$\occ_{\gs_i}(\tbtau)$ and $|\tbtau|$. 
(See also \refR{R1} when $\gs$ is a tree permutation, so $d=1$.)
We do not know any general formula,
but at least for a specific $\gs$, it should be possible to calculate these
using methods similar to those in the proof of \refLs{L8} and \ref{LAz}.
\end{remark}

\begin{example}\label{E21f}
  Consider the simplest example $\gs=21$, where we count the number of
  inversions in a random forest permutation $\bpi_n$.
In this case, $\gs$ is indecomposable, so  $d=1$.
Furthermore,  
by \eqref{f10} and \refE{E21},
\begin{align}\label{e21fa}
  f(\tau)=\occ_{21}(\tau)=|\tau|-1=h(\tau)-1,
\end{align}
and thus, using also \eqref{fnu},
\begin{align}\label{e21fgs}
\mu_{21}=\E f(\tbtau)=\nu-1=\phi+1=\phi^2,  
\end{align}
in agreement with \eqref{mugs}.
Hence, by \eqref{tmugs2} (or \eqref{tmugs}) and \eqref{ellen}, 
\begin{align}\label{x21}
\tmu_{21}=\frac{\mu_{21}}{\nu}=\frac{\phi^2}{\phi+2}
=\frac{5+\sqrt5}{10}.
\end{align}
Moreover,
\eqref{rov} yields, using also \eqref{Vtau} and \eqref{Etau},
\begin{align}\label{e21ff}
  \gamm_{21}
=\frac{1}{\nu}\Var\Bigsqpar{|\tbtau|-1-\frac{\nu-1}{\nu}|\tbtau|}
=\nu^{-3}\Var{|\tbtau|}
=\frac{3\phi^3}{(\sqrt5\phi)^3}
= 3\cdot 5^{-3/2}
\doteq 0.268.
\end{align}
Consequently,
\refT{TF} yields
\begin{align}\label{e21fc}
  \frac{\occ_{21}(\bpi_n)-\frac{5+\sqrt5}{10} n}{n\qq}
\dto \N\bigpar{0,3\cdot 5^{-3/2}}.
\end{align}

This implies also that for the case $d=2$ of \refT{TFI}, we have
$\gamm_{12}=\gamm_{21}=3\cdot5^{-3/2}$.
\end{example}

Note that $\occ_{21}(\bpi_n)$ equals the number of edges in 
the forest $G_{\bpi_n}$, and thus
$n-\occ_{21}(\bpi_n)$ is the number of components of $G_{\bpi_n}$, which
equals the number of blocks in $\bpi_n$.
Hence, \refE{E21f} implies
a central limit theorem for the number of blocks in a
random forest permutation: 

\begin{theorem}\label{Tblocks}
Let $\gbt(\bpi_n)$ be  number of blocks
in a random forest permutation $\bpi_n$,
i.e., the number of tree permutations in a decomposition \eqref{tomtam} of
$\bpi_n$.  
Then
\begin{align}\label{e21fd}
  \frac{\gbt(\bpi_n)-\frac{5-\sqrt5}{10} n}{n\qq}
\dto \N\bigpar{0,3\cdot 5^{-3/2}},
\end{align}
with convergence of all moments.
\qed
\end{theorem}

\section{Random tree permutations from random blocks}
\label{SRTB}
In the remaining sections, we study patterns in a random tree permutation
$\btau_n$.
In analogy with the construction of $\bpi_n$ from random tree permutations
$\tbtau_i$ in \refS{SRT}, we may construct the random tree permutation
$\btau_n$ with given length from a code with  blocks of random lengths.
There is only one $\sfL$-block or $\sfR$-block of each length, and 
therefore (\cf{} \eqref{tbtau}) we simply let
$(L_i)\xoo$ and $(R_i)\xoo$ 
be  two infinite sequences 
of random variables, all \iid, with
the geometric distribution
\begin{align}
  \label{LR}
\P(L_i=\ell)=\P(R_i=\ell)=2^{-\ell},
\qquad \ell\ge1.
\end{align}
We also define the random vector
\begin{align}\label{XLR}
X_i:=(L_i,R_i),   
\end{align}
and, for a vector $x=(\ell,r)$,
\begin{align}\label{hlr}
  h(x):=\ell+r.
\end{align}
We use the notation of \refS{SSU}; in particular,
\begin{align}\label{SLR}
  S_m:=\sumim h(X_i)=\sumim(L_i+R_i).
\end{align}

For $m\ge1$, let $\btaux_m$ be the random tree permution that has a code with
$2m$ blocks of lengths $L_1,R_1,\dots,L_m,R_m$, and thus (random) length
$S_m$.
Then, for every tree permutation $\tau$ having a code $\gO_\tau$ with $2m$
blocks with lengths $\ell_1,r_1,\dots,\ell_m,r_m$,
by independence and \eqref{LR},
\begin{align}\label{btaux}
  \P\bigpar{\btaux_m=\tau}
&=
\P\bigpar{L_1=\ell_1,R_1=r_1,\dots,L_m=\ell_m,R_m=r_m}
\notag\\&
=\prodim \P(L_1=\ell_i)\P(R_i=r_i)
=\prodim 2^{-\ell_i}2^{-r_i}
=2^{-|\tau|}.
\end{align}
It follows as in \refSS{SSRT}, \cf{} \eqref{piN}--\eqref{pint}, that
if $\tau$ is a tree permutation of length $n$ that has $2m$ blocks in its code,
then
\begin{align}
  \P\bigpar{\btaux_{N(n)}=\tau}
=\P\bigpar{N(n)=m \;\&\; \btaux_m=\tau}
=\P\bigpar{\btaux_m=\tau}
=2^{-|\tau|},
\end{align}
which is the same for all $\tau\in\fT_n$, and thus
\begin{align}\label{tnt}
  \btau_n\eqd\bigpar{\btaux_{N(n)}\mid S_{N(n)}=n}.
\end{align}

Note that $X_{N(n)}$ does not have the same distribution as $X_i$ for a
fixed $i$, see \eg{} \cite[Section 2.6]{Gut-SRW}.
We will use a simple (coarse) estimate (valid for much more general $X_i$ and
$h(X_i)$). Define for convenience $h(X_i):=0$ for $i\le0$.
\begin{lemma}\label{Lren}
For any $j\ge0$, $k\ge1$ and $n\ge1$,
\begin{align}\label{lren}
  \P\bigsqpar{h(X_{N(n)-j})=k}
\le\bigpar{k+j\E h(X_1)}\P\bigpar{h(X_{1})=k}.
\end{align}
Hence, for any $q>0$,
\begin{align}\label{lren2}
  \E\bigsqpar{h(X_{N(n)-j})^q}
\le \E\bigsqpar{h(X_1)^{q+1}}+j\E h(X_1) \E\bigsqpar{h(X_1)^q}
.\end{align}
\end{lemma}

\begin{proof}
Write $Y_i:=h(X_i)$ and $Z_i:=\sum_{s=1}^j Y_{i+s}$.
  If $Y_{N(n)-j}=k$, 
then there exists some $m\ge0$ (viz.\ $N(n)-j-1$) such that
$S_m<n$, $Y_{m+1}=k$, and $S_m+Y_{m+1}+Z_{m+1}\ge n$.
For a given $m$, $S_m$, $Y_{m+1}$ and $Z_{m+1}$ are independent, and thus
\begin{align}\label{lren1}
  \P\bigpar{Y_{N(n)-j}=k}
&\le \summo\sum_{i=0}^{n-1} \P\bigpar{S_m=i, \, Y_{m+1}=k, \, k+Z_{m+1}\ge  n-i}
\notag\\&
= \summo\sum_{i=0}^{n-1} \P\bigpar{S_m=i} \P\bigpar{Y_{m+1}=k}
\P\bigpar{k+Z_{m+1}\ge n-i}
\notag\\&
= \P\bigpar{Y_{1}=k}\sum_{i=0}^{n-1}\summo \P\bigpar{S_m=i} 
\P\bigpar{k+Z_{1}\ge n-i}
\notag\\&
\le\P\bigpar{Y_{1}=k} \sum_{i=0}^{n-1}\P\bigpar{k+Z_{1}\ge n-i}
\notag\\&
\le\P\bigpar{Y_{1}=k} \sum_{s=1}^{\infty}\P\bigpar{k+Z_{1}\ge s}
=\P\bigpar{Y_{1}=k} \E\bigpar{k+Z_1}
\notag\\&
=(k+j\E Y_1)\P\bigpar{Y_{1}=k}.
\end{align}
This proves \eqref{lren}.
We obtain \eqref{lren2} by
multiplying \eqref{lren} by $k^q$ and summing over $k$.
\end{proof}

We record a simple fact.
\begin{lemma}\label{Lh}
  We have $\E L_i=\E R_i=2$, and thus
  \begin{align}\label{nu4}
    \nu:=\E h(X_i)=4.
  \end{align}
\end{lemma}

\begin{proof}
By definition, $L_i\eqd R_i\sim\Ge(1/2)$, and thus, as is well known, 
$\E L_i=\E R_i=2$. (See also \eqref{lx2c} below.)
Hence,   \eqref{nu4} follows.
\end{proof}

\section{Trees in a given tree permutation}

We next 
express the number of occurences of a pattern $\gs$ in a tree
permutation using codes and block lengths.
We 
consider here only the case when $\gs$ is a tree permutation.

\begin{lemma}
  \label{LTT}
Let $\gs$ be a tree permutation with $|\gs|\ge3$ having a code with $2b$ blocks
of lengths 
$\ell_1,r_1,\dots,\ell_b,r_b$,
and let
$\tau$ be a tree permutation with $|\gt|\ge3$ having a code with $2m$
blocks of lengths
$\ell'_1,r'_1,\dots,\ell'_m,r'_m$.
Then
\begin{align}\label{ltt}
  \ns(\tau)=\sum_{s=0}^{m-b}\prod_{i=1}^{b}
\ga_{\sfL,i}\bigpar{\tell'_{i+s}}
\ga_{\sfR,i}\bigpar{\tr'_{i+s}}
\end{align}
where
\begin{align}
  \tell'_k&:=\ell'_k-\indic{k=1, \ell_1>1},\label{tell}
\\
\tr'_k&:=r'_k-\indic{k=m, r_b>1},\label{tr}
\intertext{and}
\label{lttl}
 \ga_{\sfL,i}(\ell')&:=\binom{\ell'-1+\indic{i=1,\ell_1>1}+\indic{i=b,r_b=1}}
                           {\ell_i-1+\indic{i=b, r_b=1}},
\\\label{lttr}
 \ga_{\sfR,i}(r')&:=\binom{r'-1+\indic{i=b,r_b>1}+\indic{i=1,\ell_1=1}}
                           {r_i-1+\indic{i=1, \ell_1=1}}.
\end{align}
\end{lemma}

\begin{proof}
 The occurrences of $\gs$ in $\gt$ are described by colourings of $\gO_\tau$
that can be obtained as in \refL{LC}. 
Consider one such colouring, $\hgO_\tau$ say.
We find some properties of it.

\pfitem{i}
Consider first the red symbols in $\hgO_\tau$ that correspond to a single
block $B_j$ in $\gO_\gs$. 
These red symbols have the same type ($\sfL$ of $\sfR$), and there
are no other red symbols between them.
It follows from \refL{LC} that they have to belong to the same block,
$B'_{k}$ say, in $\gt$, except for the first and last blocks $B_1$ and
$B_{2b}$. 
If $|B_1|\ge2$,  it is also possible that the first $\sfL$ in $B_1$
corresponds to the last in $B'_{k-2}$, while all others correspond to red
$\sfL$ in $B'_{k}$ (for some odd $k\ge3$). 
We have a symmetric situation for the last block $B_{2b}$ if $|B_{2b}|\ge2$.
 Write $k=k(j)$ for the index of the block $B'_k$ 
in $\hgO_\tau$ that corresponds
to $B_j$. (To be precise in all cases, 
$B'_{k(j)}$ contains the last red $\sfL$ in $B_j$
if $j$ is odd, and the first red $\sfR$ in $B_j$ if $j$ is even.)

\pfitem{ii}
Furthermore, for an $\sfL$-block $B_{2i-1}$ in $\gO_\gs$, the last $\sfL$ in the
corresponding block $B'_{k(2i-1)}$ 
in $\hgO_\tau$ 
has to be red, except in the case
of the last $\sfL$-block $B_{2b-1}$ if $|B_{2b}|=1$; in that exceptional
case there is no restriction on the red subset of $B'_{k(2b-1)}$ 
(except it having the size $\ell_{b}$ of $B_{2b-1}$). 
For an $\sfR$-block $B_{2i}$ there is a symmetric condition, unless $i=1$
and $|B_1|=1$.

\pfitem{iii}
In all cases, $k=k(j)\equiv j\pmod 2$. 
Moreover, no completely black blocks can be inserted between the red symbols
in two consecutive blocks of $\gO_\gs$. Hence,
$k(j+1)=k(j)+1$ for every $j<2b$, and thus there exists 
$s\in[0,m-b]$ such that
$k(j)=j+2s$ for all $i$.

Conversely, any choice of red symbols satisfying (i)--(iii) for some
$s\in[0,m-b]$ gives a colouring of the code $\gO_\tau$ that 
can be constructed as in \refL{LC}, and thus corresponds to an occurrence of
$\gs$ in $\tau$.

For each choice of $s$, 
the choices of red symbols permitted by (i)--(iii) for 
an  $\sfL$-block $B_{2i-1}$ 
is $\ga_{\sfL,i}(\tell'_{s+i})$;
note that for $1<i<b$, this is just $\binom{\ell'_{i+s}-1}{\ell_i-1}$,
while for $i=1$ and $b$ there are (possibly)
some adjustments that are taken care of by the indicator functions
in \eqref{tell} and \eqref{lttl}.
Similarly, the choices of red symbols for
an  $\sfR$-block $B_{2i}$ 
is $\ga_{\sfR,i}(\tr'_{s+i})$.
Hence, still for a fixed $s$, 
the total number of choices of red symbols in $\gO_\tau$
is given by the product in \eqref{ltt},
because the choices for the different blocks $B_1,\dots,B_{2b}$
can be made independently of each other.
Consequently, \eqref{ltt} holds.
\end{proof}

\begin{remark}\label{RTT}
The condition $|\gs|\ge3$ in \refL{LTT} excludes the two cases $\gs=1$ and
$\gs=21$. Recall that both these cases are trivial, with $\occ_1(\tau)=|\tau|$
and $\occ_{21}(\tau)=|\tau|-1$ for any tree permutation $\tau$.
(The latter because the number of inversions in $\tau$ equals the number of
edges in the tree $\gO_\tau$.) 
Note that $21$ has the code $LR$, so in the notation above, 
it has $b=1$ and $\ell_1=r_1=1$; however, \eqref{ltt} is not valid in this
case.
\end{remark}

Recall that $b$ in \refL{LTT} is denoted $b(\gs)$, see \refS{Sforest},
and that we also have defined $b(1):=1$ for the case $\gs=1$.
For any tree permutation $\gs$
we define, for $b=b(\gs)$
vectors $x_j=(\ell'_j,r'_j)$, 
\begin{align}\label{kem}
&\ff_\gs\bigpar{x_1,\dots,x_b}:=\prod_{i=1}^{b}
\ga_{\sfL,i}\bigpar{\ell'_{i}}
\ga_{\sfR,i}\bigpar{r'_{i}},
\qquad \text{if }|\gs|\ge3,
\intertext{
with 
$\ga_{\sfL,i}$ and $\ga_{\sfR,i}$ given by \eqref{lttl}--\eqref{lttr},
and 
}
&  \ff_\gs(x_1):=\ell'_1+r'_1=h(x_1) \label{kem2}
\qquad \text{if }|\gs|\le2.
\end{align}
(In the exceptional cases $1$ and $21$ where \eqref{kem2} applies, 
we have $b(\gs)=1$.)

We compute also some expectations needed later.
\begin{lemma}\label{Lega}
Let $\gs$ be as in \refL{LTT} and 
let $\gali$ and $\gari$ be given by \eqref{lttl}--\eqref{lttr}.
Let  $L_i$ and $R_i$ have the geometric distribution in \eqref{LR}.
Then,  
\begin{align}\label{lxl}
  \E\ga_{\sfL,i}(L_{i})&=\bigpar{1+\indic{i=1,\ell_1>1}}
\bigpar{1+\indic{i=b,r_b=1}},
\\\label{lxr}
  \E\ga_{\sfR,i}(R_{i})&=\bigpar{1+\indic{i=b,r_b>1}}
\bigpar{1+\indic{i=1,\ell_1=1}}.
\end{align}
\end{lemma}

\begin{proof}
In the definition \eqref{lttl}, there are two special cases:
(I) $i=1$ and $\ell_1>1$; (II) $i=b$ and $r_b=1$. Note that both may occur
together, if $b=1$; thus there are four possible combinations.

\pfcase{Neither (I) nor (II)} 
In this case, \eqref{lttl} is simply
$\binom{\ell'-1}{\ell_i-1}$, and thus
\begin{align}\label{lx1a}
\E  \ga_{\sfL,i}(L_{i})
=\E\binom{L_i-1}{\ell_i-1}.
\end{align}
To compute this binomial moment, we  note that the \pgf{} of $L_i-1$ is,  
by \eqref{LR},
\begin{align}\label{lx1b}
  g_{L-1}(z):=\suml z^{\ell-1}2^{-\ell}=\frac{1/2}{1-z/2}=\frac{1}{2-z},
\end{align}
and thus, 
\begin{align}\label{lx1c}
  \E \binom{L_i-1}{k}
=\frac{1}{k!}\frac{\ddx^k}{\dd z^k}g_{L-1}(1)=1,
\qquad k\ge0
.\end{align}
(Alternatively, compute $[z^k]g_{L-1}(1+z)$.)
Hence, in this case, 
\begin{align}\label{lx1d}
\E  \ga_{\sfL,i}(L_{i})=1.
\end{align}

\pfcase{(I) but not (II)} 
Then, $\ell_1\ge2$ and \eqref{lttl} yields
\begin{align}\label{lx2a}
\E  \ga_{\sfL,i}(L_{i})
=\E\binom{L_i}{\ell_i-1}.
\end{align}
The \pgf{} of $L_i$ is, by \eqref{lx1b},
\begin{align}\label{lx2b}
g_L(z)=zg_{L-1}(z)=\frac{z}{2-z}=\frac{2}{2-z}-1,
\end{align}
and thus, 
\begin{align}\label{lx2c}
  \E \binom{L_i}{k}
=\frac{1}{k!}\frac{\ddx^k}{\dd z^k}g_{L}(1)=2,
\qquad k\ge1.
\end{align}
(Alternatively, use \eqref{lx1c} and 
$\binom{L_i}{k}=\binom{L_i-1}{k}+\binom{L_i-1}{k-1}$.)
Hence, 
\eqref{lx2a} yields
\begin{align}\label{lx2d}
\E  \ga_{\sfL,i}(L_{i})=2.
\end{align}

\pfcase{(II) but not (I)} 
Then, \eqref{lttl} yields, using \eqref{lx2c},
\begin{align}\label{lx3d}
\E  \ga_{\sfL,i}(L_{i})
=\E\binom{L_i}{\ell_i}
=2.
\end{align}

\pfcase{Both (I) and (II)} 
Then, $b=i=1$, $\ell_1\ge2$, and \eqref{lttl} yields
\begin{align}\label{lx4a}
\E  \ga_{\sfL,i}(L_{i})
=\E\binom{L_i+1}{\ell_i}.
\end{align}
The \pgf{} of $L_i+1$ is, by \eqref{lx2b},
\begin{align}\label{lx4b}
g_{L+1}(z)=zg_{L}(z)=\frac{2z}{2-z}-z=\frac{4}{2-z}-2-z,
\end{align}
and thus, 
\begin{align}\label{lx4c}
  \E \binom{L_i+1}{k}
=\frac{1}{k!}\frac{\ddx^k}{\dd z^k}g_{L}(1)=4,
\qquad k\ge2
.\end{align}
(Alternatively, use \eqref{lx2c} and 
$\binom{L_i+1}{k}=\binom{L_i}{k}+\binom{L_i}{k-1}$.)
Hence, 
by \eqref{lx4a},
\begin{align}\label{lx4d}
\E  \ga_{\sfL,i}(L_{i})=4.
\end{align}

We may summarize the four cases \eqref{lx1d}, \eqref{lx2d}, \eqref{lx3d}
and \eqref{lx4d} as \eqref{lxl}.
Similarly, by only notational changes, \eqref{lttr} yields
\eqref{lxr}.
\end{proof}

\begin{lemma}\label{Laga}
  Let $\gs$ be any tree permutation, let $b:=b(\gs)$,
and let $(X_i)_i$ be the \iid{} random vectors defined in \eqref{XLR}.
Then
\begin{align}\label{lj3}
  \E \fgs\bigpar{X_1,\dots,X_b}=4.
\end{align}
\end{lemma}

\begin{proof}
The case $|\gs|\le2$ is immediate by \eqref{kem2} and \refL{Lh}.

Assume thus $|\gs|\ge3$.
Then, by \eqref{kem}, independence, and \refL{Lega},
\begin{align}
&  \E \fgs(X_1,\dots,X_b)
= \E \prodib \ga_{\sfL,i}\bigpar{L_{i}}\ga_{\sfR,i}\bigpar{R_{i}}
=  \prodib \E\ga_{\sfL,i}\bigpar{L_{i}}\prodib\E \ga_{\sfR,i}\bigpar{R_{i}}
\notag\\&\quad=
\bigpar{1+\indic{\ell_1>1}}
\bigpar{1+\indic{r_b=1}}
\cdot\bigpar{1+\indic{r_b>1}}
\bigpar{1+\indic{\ell_1=1}}
\notag\\&\quad=
\bigpar{1+\indic{\ell_1>1}}
\bigpar{1+\indic{\ell_1=1}}
\cdot\bigpar{1+\indic{r_b=1}}
\bigpar{1+\indic{r_b>1}}
  \notag\\&\quad=
(1+1)(1+1)=4,
\end{align}
 which completes the proof.
\end{proof}

We have no simple explanation for the, perhaps surprising, fact that the
expectation \eqref{lj3} is the same for every tree permutation $\gs$,
\cf~\refProb{Prob=}.

\section{Patterns in a random tree permutation
of given length}\label{Staun}
We next consider the occurrences of a pattern $\gs$ in a random
tree permutation $\btau_n$. 
We use the construction and notation in \refSs{SRTB} and \ref{SSU}.
In particular, $X_n$ and $S_m$ are defined by \eqref{LR}--\eqref{SLR}
and $N(n)$ by \eqref{Nx}.

We first consider the case of a tree permutation $\gs$.
Recall $\fgs$ defined by \eqref{kem}--\eqref{kem2}.
\begin{lemma}\label{LJ}
 Let $\gs$ be a tree permutation 
and let $b:=b(\gs)$.
Then
\begin{align}\label{lj1}
  \ns(\btau_n)\eqd
\Bigpar{\sum_{s=0}^{N(n)-b}
\fgs\bigpar{X_{s+1},\dots,X_{s+b}}
\Bigm| S_{N(n)}=n}
+\OLX(1)
.\end{align}
\end{lemma}

\begin{proof}
Assume first $|\gs|\ge3$, so $\fgs$ is given by \eqref{kem}.
Recall $\btaux_m$ defined in \refS{SRTB}, and note that
$\btaux_{N(n)}$ is a tree permutation having a code with $2N(n)$ blocks of
lengths $L_1,\dots,R_{N(n)}$. \refL{LTT} thus shows that
\begin{multline}\label{lja}
  \ns\bigpar{\btaux_{N(n)}} = 
\sum_{s=0}^{N(n)-b}\prod_{i=1}^{b}
\ga_{\sfL,i}\bigpar{L_{i+s}-\indic{s=0,i=1,\ell_1>1}}
\\\cdot
\ga_{\sfR,i}\bigpar{R_{i+s}-\indic{s=N(n)-b, i=b, r_b>1}}.
\end{multline}
Except in the extreme cases $s=0$ and $s=N(n)-m$,
 the product in the sum in \eqref{lja}
is
\begin{align}\label{ljb}
\prodib \ga_{\sfL,i}\bigpar{L_{i+s}}\ga_{\sfR,i}\bigpar{R_{i+s}}
=\fgs \bigpar{X_{s+1},\dots,X_{s+b}}
.\end{align}
In the cases $s=0$ and $s=N(n)-b$, the product might be smaller, but is still
$\ge0$.
Hence, \eqref{lja} yields
\begin{align}
  \label{ljc}
\sum_{s=1}^{N(n)-b-1}
\fgs\bigpar{X_{s+1},\dots,X_{s+b}}
\le  \ns\bigpar{\btaux_{N(n)}} 
\le
\sum_{s=0}^{N(n)-b}
\fgs\bigpar{X_{s+1},\dots,X_{s+b}}
.\end{align}
We claim that 
\begin{align}\label{ljclaim}
\fgs\bigpar{X_{1},\dots,X_{b}}
,\;
\fgs\bigpar{X_{N(n)-b+1},\dots,X_{N(n)}}
=\OLX(1).
\end{align}
This implies that the difference of the first and last sums in \eqref{ljc}
is $\OLX(1)$, and thus 
\begin{align}
  \label{ljc2}
  \ns\bigpar{\btaux_{N(n)}} 
=
\sum_{s=0}^{N(n)-b}
\fgs\bigpar{X_{s+1},\dots,X_{s+b}}
+\OLX(1)
.\end{align}

To show \eqref{ljclaim}, note first that \eqref{kem} and
\eqref{lttl}--\eqref{lttr} imply that
\begin{align}\label{bam1}
  \fgs(X_{k+1},\dots,X_{k+b})\le  \prodib (L_{k+i}+R_{k+i})^c
=\prodib h(X_{k+i})^c,
\end{align}
for some $c<\infty$ depending on $\gs$ only.
Hence, using \Holder's inequality, \eqref{ljclaim} follows if we show that
for every $q<\infty$ and every $j\in[1,b]$,
\begin{align}\label{bam2}
  \E h(X_j)^q &=O(1),&
  \E h(X_{N(n)-b+j})^q &=O(1),
\end{align}
The first part is trivial, since for any fixed $j$, we have 
$\E h(X_j)^q = \E h(X_1)^q<\infty$.
The second part follows from \refL{Lren}.

Hence, \eqref{ljc2} holds, and \eqref{lj1} follows by conditioning on
$S_{N(n)}=n$, recalling \eqref{tnt}. Note that the error term $\OLX(1)$
survives
this conditioning, because $\P(\SNn=n)\to 1/\E h(X_1)>0$,
see \eg{} \cite[Theorem 2.4.2]{Gut-SRW}, 
and thus for any $q<\infty$,
\begin{align}\label{cq}
  \E\bigsqpar{ |\OLX(1)|^q\mid \SNn=n}
\le
\frac{  \E\bigsqpar{ |\OLX(1)|^q}}{\P\bigsqpar{\SNn=n}}=O(1).
\end{align}

Finally, if $|\gs|\le2$, then $b=1$ and 
\begin{align}
\sum_{s=0}^{N(n)-b}\fgs(X_{s+1})  
=
\sum_{s=0}^{N(n)-1}h(X_{s+1}) =S_{N(n)}. 
\end{align}
Furthermore, $\ns(\btau_n)=n$ or $n-1$, and thus \eqref{lj1} is trivial.
\end{proof}

The sum in \eqref{lj1} is a constrained \Ustat{} of the type in \eqref{UU1},
with $d=1$ and $b_1=b(\gs)$.
We extend \refL{LJ} to forest permutations $\gs$.

\begin{lemma}\label{LK}
  Let $\gs$ be a forest permutation with block decomposition
  $\gs=\gs_1\opluss\gs_d$.
Let $b_j:=b(\gs_j)$, 
and 
define
\begin{align}\label{lk0}
  \fgs\bigpar{(x_{1i})_{i=1}^{b_1}, \dots,(x_{di})_{i=1}^{b_d}}
:=\prod_{j=1}^d \ff_{\gs_j}\bigpar{x_{j1},\dots,x_{jb_j}}.
\end{align}
Then
\begin{align}\label{lk}
  \ns(\btau_n)\eqd
\Bigpar{\UU_{N(n)}(\fgs)
\Bigm| S_{N(n)}=n}
+\OLX(n^{d-1})
.\end{align}
\end{lemma}
\begin{proof}
Recall again $\btaux_m$ from \refS{SRTB}, and
consider first $\ns(\btaux_m)$, for some given $m$.
By definition, $\btaux_m$ has $2m$ blocks, which we denote by
$B'_1,\dots,B'_{2m}$. 

As before, we mark an occurrence of $\gs$ in $\tau=\btaux_m$ by colouring
the corresponding symbols in the code $\gO_\tau$ red (and the remaining ones
black). Then each $\gs_j$ corresponds to a set of red symbols, $A_j$ say;
these sets $A_j$ are subsets of $\set{1,\dots,|\btau_m|}$.

As in \eqref{mbb}, let $\mbb_j:=b_j-1$.
For each $\gs_j$ with $|\gs_j|\ge3$, the red symbols $A_j$ are as in the
proof of \refL{LTT}, and they lie in some blocks
$B'_{2i_j-1},\dots,B'_{2(i_j+\mbb_j)}$, possibly also with a red symbol in
$B'_{2i_j-3}$ or $B'_{2(i_j+\mbb_j)+2}$.

If $|\gs_j|=2$, so $\gs_j=21$, then the red symbols in $A_j$ are an $\sfL$
and an $\sfR$ forming an edge, and thus described by
\refL{Lleaf}\ref{e1}--\ref{e3}; we then define $i_j$ 
so that the $\sfL$ belong to $B'_{2i_j-1}$ (and thus the $\sfR$ to
$B'_{2i_j}$ or $B'_{2i_j+2}$).

Finally, if $\gs_j=1$, $A_j$ is a single red symbol, which can be either
$\sfL$ or $\sfR$; we define $i_j$ such that this symbol belongs to
$B'_{2i_j-1}$ or $B'_{2i_j}$.

The sets $A_j$ follow each other in order, and thus
we must have $1\le i_1\le i_2\le \dots i_d\le m$. (Equality is possible,
e.g.\ if $|\gs_j|=1$.)
Moreover, 
for a given sequence $i_1,\dots,i_d$,
if all gaps $i_{j+1}-i_j\ge3$, then the sets $A_j$ can be chosen
independently, without interfering with each other (by colliding, having
symbols in wrong order, or causing edges between two of them).
If furthermore $i_1>1$ and $i_d+\mbb_d<m$, the number of choices for each
$\gs_j$ 
with $|\gs_j|\ge3$ is $\ff_{\gs_j}\bigpar{X_{i_j},\dots,X_{i_j+\mbb_j}}$
by the proof of \refL{LTT}. 
The same holds for $|\gs_j|\le2$ by the definition \eqref{kem2}: 
if $\gs_j=1$, then $A_j$ is one of the $L_{i_j}+R_{i_j}$ symbols in 
$B'_{2i_j-1}\cup B'_{2i_j}$;
if $\gs=21$, then $A_j$ consists of an $\sfL$ in 
$B'_{2i_j-1}$ and an $\sfR$ in $B'_{2i_j}$  or $B'_{2i_j+2}$ 
chosen according to one of
\ref{e1}--\ref{e3} in \refL{Lleaf}, and this too gives $L_{i_j}+R_{i_j}$
choices. (Note that \ref{e1} and \ref{e2} overlap in one possibility.)
Hence, for such $i_1,\dots,i_d$ the number of possible choices of 
$ A_1,\dots,A_d$ is
\begin{align}\label{columbus}
\prod_{j=1}^d \ff_{\gs_j}\bigpar{X_{i_j},\dots,X_{i_j+\mbb_j}}
=  \fgs\bigpar{ (X_{i})_{i=i_1}^{i_1+\mbb_1}, \dots, (X_i)_{i=i_d}^{i_d+\mbb_d}}.
\end{align}
If some gap $i_{j+1}-i_j\le2$, the number of possibilities may be smaller,
but we may conclude that, recalling the definition \eqref{UU1},
\begin{align}\label{lars}
\bigabs{\ns\bigpar{\btaux_m}-\UU_m(\fgs)}
\le
\sumx \fgs\bigpar{ (X_{i})_{i=i_1}^{i_1+\mbb_1}, \dots, (X_i)_{i=i_d}^{i_d+\mbb_d}}
,\end{align}
where $\sumx$ denotes the sum over $i_1,\dots,i_d\in[1,m-\mbb_d]$ such that either 
$i_1=1$, $i_d=m-\mbb_d$, or $i_j\le i_{j+1}\le i_j+2$ for some $j$. 

We now take $m=N(n)$, condition on $\SNn=n$ and use \eqref{tnt}.
It remains only to show that the sum in \eqref{lars}
(with $m=N(n)$) is $\OLX(n^{d-1})$; this then survives the conditioning as
in \eqref{cq}.
To see this, consider first the terms with $i_1=1$ or $i_j\le i_{j+1}\le i_j+2$ 
for some $i$. Since $m=N(n)\le n$, we may extend the sum to all
$i_1,\dots,i_d\in[1,n]$ satisfying one of these conditions. This is a sum of
$O(n^{d-1})$ terms, and each term is $\OLX(1)$ by 
\eqref{lk0},
\eqref{bam1}--\eqref{bam2}
and \Holder's inequality. 
Hence the sum of these terms is $\OLX(n^{d-1})$ by Minkowski's inequality.

The remaining sum consists of terms with $i_d=m-\mbb_d=N(n)-\mbb_d$, and is
thus 
\begin{align}\label{ola}
  \le \ff_{\gs_d}\bigpar{X_{N(n)-\mbb_d},\dots,X_{N(n)}}
\sum_{i_1,\dots,i_{d-1}=1}^{n}\prod_{j=1}^{d-1}\ff_{\gs_j}\bigpar{X_{i_j},\dots,X_{i_j+\mbb_j}}
.
\end{align}
The first factor is $\OLX(1)$ as shown in \eqref{bam1}--\eqref{bam2}, and
the sum is again a sum of $O(n^{d-1})$ terms that are $\OLX(1)$,
and thus this sum is $\OLX(n^{d-1})$ by Minkowski's inequality.
Hence, \eqref{ola} is $\OLX(n^{d-1})$ by \Holder's inequality, which
completes the proof.
\end{proof}

\begin{proof}[Proof of \refT{TT}]
\refL{LK} and \refP{PUUN2}
show that
\begin{equation}\label{tt1}
\frac{\ns(\btau_n)-\mu{\mux}^{-d}{d!}\qw n^{d}}
{n^{d-1/2}} 
\dto \N\bigpar{0,\gamx^2},
\end{equation}
with convergence of all moments by \refP{PUUNmom};
note that $\E|\ff(X_1,\dots,X_D)|^p<\infty$ 
and $\E h(X_1)^p<\infty$
for every $p<\infty$ by
\eqref{lk0},  \eqref{ljclaim}, and \eqref{bam2}.
Furthermore,
\eqref{lk0} and \refL{Laga} imply that
\begin{align}\label{tt2}
  \mu=\E\ff(X_1,\dots,X_D)=\prod_{j=1}^d \E\ff_{\gs_j} (X_1,\dots,X_{b_j})
=4^d,
\end{align}
while $\nu=4$ by \refL{Lh}.
Hence, $\mu\nu^{-d}=1$, and \eqref{tt} follows from \eqref{tt1}.

To see that $\gamm>0$ if some $|\gs_j|\ge3$,
  we use the criterion in \refP{PUU0}.
By \eqref{fj} and \eqref{lk0},
\begin{align}\label{ek}
  \ff_{\gs,j}(x_1,\dots,x_{b_j})
=\ff_{\gs_j}(x_1,\dots,x_{b_j})\prod_{i\neq  j} \E\ff_{\gs_i}\bigpar{X_1,\dots,X_{b_i}}
= c\ff_{\gs_j}(x_1,\dots,x_{b_j})
\end{align}
for some constant $c>0$.
Now suppose that
 $|\gs_j|\ge3$.
Then, \eqref{ek},
\eqref{kem} and \eqref{lttl}--\eqref{lttr} show that, 
with $x_i=(\ell'_i,r'_i)$, 
$\ff_{\gs,j}(x_1,\dots,x_{b_j})$ is a polynomial in $\set{\ell'_i,r'_i}$ of
total degree
\begin{align}\label{svante}
\gd_j:=  \sum_{i=1}^{b_j} \bigpar{\ell_i-1+r_i-1} + \indic{r_{b_j}=1}+\indic{\ell_1=1}.
\end{align}
We see also that the polynomial has only one term with this degree, and that
this term has a positive coefficient.
Note further that 
\begin{align}\label{sturar}
\gd_j \ge \bigpar{\ell_1-1+\indic{\ell_1=1}}
+\bigpar{r_{b_j}-1+ \indic{r_{b_j}=1}}
\ge 2.
\end{align}
In particular, if we take $x_1=\dots=x_{b_j}=(s,s)$, then 
$\ff_{\gs,j}(x_1,\dots,x_{b_j})$ is a polynomial in $s$
of degree $\gd_j\ge2$.
Hence, if we fix any $n>2b_j$, and consider the event
(which has positive probability)
\begin{align}
  X_i=(L_i,R_i)=
  \begin{cases}
    (s,s),& b_j<i\le 2b_j,
\\ (1,1), & \text{otherwise}
  \end{cases}
\end{align}
for an integer $s\ge1$, we see that 
$S_n(\ff_{\gs,j})$ defined in \eqref{emma}
is a polynomial in $s$ of degree $\gd_j\ge2$.
Furthermore, on the same event, $S_n(h)$ is a polynomial in $s$ of degree 1,
and thus, 
$S_n\bigpar{\ff_{\gs,j}-\frac{\mu}{\nu}h}$ is a non-constant polynomial in $s$.
Consequently, the condition in \refP{PUU0} cannot be satisfied for 
$\ff_{\gs,j} -\frac{\mu}{\nu}h$, and thus \refP{PUUN} shows that $\gamm>0$.
\end{proof}

We compute the asymptotic variance $\gamm$ only in a simple special case.

\begin{example}\label{Elr}
  Suppose that $\gs$ is a tree permutation with $b(\gs)=1$;
thus its code has only two blocks, of lengths $\ell_1=\ell$ and $r_1=r$.
Then, the \Ustat{} $\UU_N(\ff_\gs)$ in \eqref{lk} is simply a partial sum:
\begin{align}\label{sten}
  \UU_N(\ff_\gs)=\sum_{i=1}^N \ff_\gs(X_i) = S_N(\ff_\gs).
\end{align}
This is the special case $d=1$ of an unconstrained \Ustat{}
discussed in \refR{R1}, and \eqref{rov} yields, 
since $\mu=\nu=4$ by \eqref{tt2} and \eqref{nu4}, 
\begin{align}\label{sture}
  \gamm&=\frac{1}{4}\Var\bigsqpar{\ff_\gs(X)-h(X)}
\notag\\&
=\frac{1}{4}\Bigpar{\Var\bigsqpar{\ff_\gs(X)}-2\Cov\bigsqpar{\ff_\gs(X),h(X)}
+\Var\bigsqpar{ h(X)}},
\end{align}
where $X=(L,R)$ with independent $L,R\sim\Ge(1/2)$  as in
\eqref{LR}--\eqref{XLR}.
We recall that $\E L=\E R=2$. A simple calculation, for example using
\eqref{lx2c}, yields $\Var L = \Var R = 2$
and thus $\Var h(X)=\Var(L+R)=4$.

We consider several cases.

\resetCase
\pfcase{$\ell=r=1$} 
This means $\gO_\gs=\sfL\sfR$, and thus $\gs=21$.
As we have seen earlier, this case is trivial
and  $\ns(\btau_n)$ is deterministic.
Indeed, we have $\ff_\gs=h$ and thus \eqref{sture} yields $\gamm=0$.

\pfcase{$\ell>1$, $r=1$} 
This means that $\gs$ is the permutation $23\cdots(\ell+1)1$.

By \eqref{kem} and \eqref{lttl}--\eqref{lttr},
\begin{align}
  \ff_\gs(L,R)=\ga_{\sfL,1}(L)\ga_{\sfR,1}(R)
=\binom{L+1}{\ell}\binom{R-1}{0}=\binom{L+1}{\ell}.
\end{align}
We have, using \eqref{lx4c},
\begin{align}
\E\lrsqpar{L\binom{L+1}{\ell}}=(\ell-1)\E\binom{L+1}{\ell}+(\ell+1)\E\binom{L+1}{\ell+1}=8\ell,  
\qquad \ell\ge2,
\end{align}
and thus \eqref{sture} yields, using also $\E L^2=6$,
\begin{align}
  \gamm&=\frac{1}{4}\Var\lrsqpar{\binom{L+1}{\ell}-(L+R)}
=\frac{1}{4}\lrpar{\Var\lrsqpar{\binom{L+1}{\ell}-L}+2}
\notag\\&
=\frac{1}{4}\lrpar{\E\lrsqpar{\lrpar{\binom{L+1}{\ell}-L}^2}-2}
=\frac{1}{4}{\E\lrsqpar{\binom{L+1}{\ell}^2}}-4\ell+1
.\end{align}
This can easily be evaluated for any $\ell\ge2$, although we do not know a
closed formula.

\pfcase{$\ell=1$, $r>1$} 
This means that $\gs$ is the permutation
$(r+1)1\cdots r\in \fT_{r+1}$.
This case is the same as the preceding one, if we exchange 
$\ell\leftrightarrow r$ and $L \leftrightarrow R$.

\pfcase{$\ell>1$, $r>1$}
This means that $\gs=2\cdots\ell(\ell+r)1(\ell+1)\cdots(\ell+r-1)$.
By \eqref{kem} and \eqref{lttl}--\eqref{lttr},
\begin{align}
  \ff_\gs(L,R)=\ga_{\sfL,1}(L)\ga_{\sfR,1}(R)
=\binom{L}{\ell-1}\binom{R}{r-1}
.\end{align}
We have, using \eqref{lx2c},
\begin{align}
\E\lrsqpar{L\binom{L}{\ell-1}}
=(\ell-1)\E\binom{L}{\ell-1}+\ell\E\binom{L}{\ell}=4\ell-2,  
\qquad \ell\ge2,
\end{align}
and thus \eqref{sture} yields
\begin{align}
  \gamm&=\frac{1}{4}\E\lrsqpar{\lrpar{\binom{L}{\ell-1}\binom{R}{r-1}-L-R}^2}
\notag\\&
=\frac{1}{4}\E\lrsqpar{\binom{L}{\ell-1}^2}\E\lrsqpar{\binom{R}{r-1}^2}
-\frac12\E\lrsqpar{\binom{L}{\ell-1}\binom{R}{r-1}(L+R)}
+\frac14\E(L+R)^2
\notag\\&
=\frac{1}{4}\E\lrsqpar{\binom{L}{\ell-1}^2}\E\lrsqpar{\binom{R}{r-1}^2}
-4(\ell+r)+9
.\end{align}
Again, this is easily evaluated for any $\ell,r\ge2$.

Some numerical values for small $\ell$ and $r$ are given in
\refTab{tab:Elr}.
These values are integers
(but they do not seem to correspond to any integer sequence in \cite{OEIS});
we conjecture that $\gamm(\ell,r)$ is an integer for all
$\ell,r\ge1$, but we have no proof.

Note that $\gamm(1,3)\neq\gamm(2,2)$, which verifies our claim after
\refC{CTT} that $\gamm_\gs$ can differ for different tree permutations
$\gs$, even if they have the same length.
\end{example}

\begin{problem}\label{Plr}
  In \refE{Elr}, is $\gamm$ an integer for every $\ell,r\ge1$?
\end{problem}

\begin{problem}\label{PZ}
  Is $\gamm_\gs$ an integer for every tree permutation $\gs$?
For every forest permutation $\gs$?
\end{problem}


\begin{table}[ht]
  \centering
\begin{tabular}{r||r|r|r|r|r}
 $\ell\backslash r$ & 1 & 2 & 3 & 4 & 5\\
\hline\hline
1 & 0 & 6 & 52 & 306 & 1664\\
\hline
2 & 6 & 2 & 28 & 174 & 944\\
\hline
3 & 52 & 28 & 154 & 800 & 4150\\
\hline
4 &  306 & 174 & 800 & 3946 & 20196\\
\hline
5 & 1664 & 944 & 4150 & 20196 & 103010  
\end{tabular}

  \caption{Some numerical values of $\gamm=\gamm(\ell,r)$ in \refE{Elr}.}
  \label{tab:Elr}
\end{table}

\newcommand\AAP{\emph{Adv. Appl. Probab.} }
\newcommand\JAP{\emph{J. Appl. Probab.} }
\newcommand\JAMS{\emph{J. \AMS} }
\newcommand\MAMS{\emph{Memoirs \AMS} }
\newcommand\PAMS{\emph{Proc. \AMS} }
\newcommand\TAMS{\emph{Trans. \AMS} }
\newcommand\AnnMS{\emph{Ann. Math. Statist.} }
\newcommand\AnnPr{\emph{Ann. Probab.} }
\newcommand\CPC{\emph{Combin. Probab. Comput.} }
\newcommand\JMAA{\emph{J. Math. Anal. Appl.} }
\newcommand\RSA{\emph{Random Structures Algorithms} }
\newcommand\DMTCS{\jour{Discr. Math. Theor. Comput. Sci.} }

\newcommand\AMS{Amer. Math. Soc.}
\newcommand\Springer{Springer-Verlag}
\newcommand\Wiley{Wiley}

\newcommand\vol{\textbf}
\newcommand\jour{\emph}
\newcommand\book{\emph}
\newcommand\inbook{\emph}
\def\no#1#2,{\unskip#2, no. #1,} 
\newcommand\toappear{\unskip, to appear}

\newcommand\arxiv[1]{\texttt{arXiv}:#1}
\newcommand\arXiv{\arxiv}

\newcommand\xand{and }
\renewcommand\xand{\& }

\def\nobibitem#1\par{}

\end{document}